\documentclass[12pt, a4paper]{amsart}

\usepackage{amsmath, amssymb, amsfonts, amsthm}
\usepackage[margin=2.0cm]{geometry}
\usepackage[]{setspace}
\usepackage[usenames,dvipsnames]{color}
\usepackage[,unicode]{hyperref}
\hypersetup{
    urlcolor=blue,
    colorlinks=true,
    linkcolor= blue,
    citecolor=blue,
}
\usepackage{verbatim}
\usepackage{fixmath}
\usepackage[all]{xy}

\usepackage{enumitem}
\setlist[enumerate]{label={\bf(\roman*)}, itemsep=1ex,leftmargin=1.4cm,topsep=1ex}
\setlist[enumerate,2]{label={\bf\alph*}., itemsep=1ex,leftmargin=0.5cm,topsep=1ex}
\setlist[enumerate,3]{label={\bf\roman*}., itemsep=1ex,leftmargin=0.5cm,topsep=1ex}

\usepackage[T1]{fontenc}
\usepackage[utf8]{inputenc}

\theoremstyle{plain}
\newtheorem{theorem}{Theorem}[section]
\newtheorem{proposition}[theorem]{Proposition}
\newtheorem{lemma}[theorem]{Lemma}
\newtheorem{corollary}[theorem]{Corollary}
\newtheorem{claim}{Claim}[theorem]

\theoremstyle{definition}
\newtheorem{example}[theorem]{Example}
\newtheorem{definition}[theorem]{Definition}

\theoremstyle{remark}
\newtheorem{remark}[theorem]{Remark}

\newenvironment{claimproof}[1][Proof of claim]
  {%
    \proof[#1]%
  }
  {%
    \endproof%
  }

\usepackage{tikz}
\usetikzlibrary{shapes.geometric, arrows}
\tikzstyle{startstop} = [rectangle, rounded corners, minimum width=3cm, minimum height=1cm,text centered, draw=black, fill=red!30]
\tikzstyle{io} = [trapezium, rounded corners, trapezium left angle=70, trapezium right angle=110, minimum width=3cm, minimum height=1cm, text centered, draw=black, fill=blue!30]
\tikzstyle{process} = [rectangle, rounded corners, minimum width=3cm, minimum height=1cm, text centered, draw=black, fill=orange!30]
\tikzstyle{decision} = [rectangle, rounded corners, minimum width=3cm, minimum height=1cm, text centered, draw=black, fill=green!30]
\tikzstyle{arrow} = [thick,->,>=stealth]

\newcommand\res{\mathrm{res}}

\providecommand{\bfbeta}{\mathbold{\beta}}

\providecommand{\rmGamma}{\mathrm{\Gamma}}

\providecommand{\rmPhi}{\mathrm{\Phi}}
\providecommand{\rmPsi}{\mathrm{\Psi}}

\title[The model theory of Cohen rings]{\Large\rm The model theory of Cohen rings}
\author{Sylvy Anscombe and Franziska Jahnke}
\address{Universit\'{e} de Paris and Sorbonne Universit\'{e}, CNRS, IMJ-PRG, F-75006 Paris, France}
\email{sylvy.anscombe@imj-prg.fr}
\address{Institut f\"{u}r Mathematische Logik und Grundlagenforschung,
University of M\"{u}nster,
Einsteinstr. 62,
48149 M\"{u}nster,
Germany}
\email{franziska.jahnke@wwu.de}
\thanks{\today}

\begin{document}
\begin{abstract}
The aim of this article is to give a self-contained account of the algebra and model theory of Cohen rings,
a natural generalization of Witt rings.
Witt rings are only valuation rings in case the residue field is perfect,
and Cohen rings arise as the Witt ring analogon over imperfect residue fields.
Just as one studies truncated Witt rings to understand Witt rings,
we study Cohen rings of positive characteristic as well as of characteristic zero.
Our main results are a relative completeness and a relative model completeness result for Cohen rings,
which imply the corresponding Ax--Kochen/Ershov type results for unramified henselian valued fields also in case the residue field is imperfect.
\end{abstract}
\maketitle

\section{Introduction}

The aim of this paper is to give an introduction to the model theory of complete Noetherian local rings $A$ which have maximal ideal $pA$.
From an algebraic point-of-view, the theory of such rings is classical.
Under the additional hypothesis of regularity, they are valuation rings,
and their study goes back to work of Krull (\cite{Kru37}) and many others.
Structure theorems were obtained by Hasse and Schmidt (\cite{HS34}), although there were deficiencies in the case that $A/pA$ is not perfect.
Further structural results were obtained by Witt (\cite{Wit37})
and Teichm\"{u}ller (\cite{Tei36b}).
In particular Teichm\"{u}ller gave a brief but precise account of the structure of such rings, even in the case that $A/pA$ is imperfect.
This was followed by Mac Lane (\cite{Mac39c}), who improved upon Teichm\"{u}ller's theory and proved relative structure theorems. Mac Lane built his work upon his study of Teichm\"{u}ller's notion of $p$-independence
in \cite{Tei36a}.
For further historical information,
especially on this early period,
the reader is encouraged to consult Roquette's article \cite{Roq03} on the history of valuation theory.

Turning away from the hypothesis of regularity, Cohen (\cite{Coh46}) gave an account of the structure of such rings.
In fact his context was even more general: he did not assume Noetherianity.

Despite all of this work, more modern treatments
(e.g.~Serre, \cite{Ser79})
of this subject are often restricted to the case that $A/pA$ is perfect.
Consequently, the literature on the model theory of complete Noetherian local rings is sparse.
For example, \cite{vdD14} also assumes that $A/pA$ is perfect.

We became interested in the model theory of 
complete Noetherian local rings when we 
started to construct examples of
NIP henselian valued fields with imperfect residue field in order to obtain an understanding
of NIP henselian valued fields (\cite{AJ3}). After getting
acquainted with the algebra of these rings as scattered in the literature
detailled above,
we realized that with a bit of tweaking, the proof ideas of these (classical) 
results can be used gain an understanding of the model theory of such rings.
To start with, this requires a careful recapitulation of the known 
algebraic (or structural) theory of such rings, bringing older results 
together in one framework.
This overview is  given in Part I of the article.
In this first part, many of the proof
ideas are inspired by the work of others (and we point to the original 
sources), but
we take care to prove everything which cannot be cited directly from elsewhere.

The underlying definition of a Cohen ring is the following:
\begin{definition}[{Cf.~Definitions \ref{def:preCohen.ring} and \ref{def:Cohen.ring}}]
A {\bf Cohen ring} is a complete Noetherian local ring $A$ with maximal ideal $pA$, where $p$ is the residue characteristic of $A$.
\end{definition}
A Cohen ring may either have characteristic $0$ (in which case we call it 
strict) or $p^n$, where $p$ is the characteristic of the residue field
$A/pA$.
In the second section, we introduce Cohen rings and recall that, for a given field $k$ of positive characteristic, Cohen rings of every
possible characteristic exist, which have residue field $k$.
In the third section, we discuss and develop the machinery of multiplicative
representatives,
namely good sections of the residue map from the perfect core of the residue field
into the Cohen ring $A$
(Definition~\ref{def:Teichmuller}).
We also comment on the extent to which
these sections are unique,
see Theorem~\ref{thm:representatives},
and how one can use them to generate Cohen rings (Proposition~\ref{prp:generation}).
In the fourth section, we prove that Teichm\"uller's embedding technique
works in this context: we embed a Cohen ring with residue
field $k$, with a choice of representatives, into the corresponding
Cohen ring over the perfect hull of $k$ 
(see Theorem~\ref{thm:TEP}).
Building on this and using ideas from Cohen, we show that any two
Cohen rings of the same characteristic and over the same residue field, both
equipped with representatives,
are isomorphic.
In fact there is a unique isomorphism which respects the choices
of representatives and is the identity on the residue field (Cohen
Structure Theorem, \ref{cor:Cohen_structure_1}).
In the final section of the first part of the paper, we compare Cohen rings to Witt rings.

In the second part we begin a model-theoretic study, including describing the complete theories of Cohen ring of a fixed characteristic, over a given residue field.
We work in the language $\mathfrak{L}_{\mathrm{vf}} = \mathfrak{L}_\mathrm{ring}\cup \{\mathcal{O}\}$ of valued fields.
We then show relative completeness, using a classical proof strategy together with the
Cohen Structure Theorem from section 6.
In particular, this result gives the following Ax--Kochen/Ershov 
principle:

\begin{theorem}[{Cf.~Corollary \ref{cor:AKE}}]\label{thm:intro_AKE}
Let $(K,v)$ and $(L,w)$ be two unramified henselian valued fields.
Then
$$
\underbrace{Kv \equiv Lw}_{\textrm{in }\mathfrak{L}_\mathrm{ring}}  \textrm{ and }  \underbrace{vK \equiv wL}_{\textrm{in }\mathfrak{L}_\mathrm{oag}} \,
\Longleftrightarrow \, 
\underbrace{(K,v)\equiv (L,w)}_{\textrm{in }\mathfrak{L}_\mathrm{vf}}.
$$
\end{theorem}

Note that this was already claimed by B\'elair in \cite[Corollaire 5.2(1)]{Bel99}.
However, since his proof crucially relies on Witt rings, it only works
for perfect residue fields.

Moreover, we prove the following relative model-completeness result, which again essentially
builds on the Cohen Structure Theorem.
\begin{theorem}[{Cf.~Corollary \ref{cor:AKE-E}}]\label{thm:intro_AKE_MC}
Let $(K,v)\subseteq (L,w)$ be two unramified henselian valued fields.
Then, we have
$$ \underbrace{Kv \preceq Lw}_{\textrm{in }\mathfrak{L}_\mathrm{ring}} \textrm{ and } \underbrace{vK \preceq wL}_{\textrm{in }\mathfrak{L}_\mathrm{oag}} \, \Longleftrightarrow \, 
\underbrace{(K,v) \preceq (L,w)}_{\textrm{in }\mathfrak{L}_\mathrm{vf}}.
$$
\end{theorem}

In the penultimate section, we prove an
embedding lemma (Proposition \ref{prp:emb}) similar to one proved by Kuhlmann for tame fields in \cite{FVK}. This is the most delicate proof in the model-theoretic part of the paper. We then apply the embedding lemma to show relative existential closedness of unramified henselian valued fields
assuming that the residue fields have a fixed finite degree of imperfection (Theorem \ref{thm:AKEE}).

Finally, applying the embedding lemma once again, we argue that in any unramified henselian valued field,
there is no new structure induced on the residue field and value group:
\begin{theorem}[Cf.~Theorem \ref{thm:SEk}]
Let $(K,v)$ be an unramfied henselian valued field. Then the value group $vK$ and the residue field $Kv$ are both stably embedded, as a pure ordered abelian group and as a pure field, respectively.
\end{theorem}
We conclude by giving an example of a finitely ramified henselian valued field in which the residue field is no longer stably embedded (Example \ref{ex:fr}).

\part{The structure of Cohen rings}
\label{part:1}

\section{Pre-Cohen rings and Cohen rings}
\label{section:Cohen}

Throughout this paper, $A,B,C$ will denote rings, which will always have a multiplicative identity $1$ and be commutative;
and $k,l$ will be fields of characteristic $p$, which is a fixed prime number.

A ring $A$ is {\bf local}
if it has a unique maximal ideal,
which we will usually denote by $\mathfrak{m}$.
A local ring is equipped with the {\bf local topology}%
\footnote{The local topology is also known as the {\em $\mathfrak{m}$-adic topology}.},
which is the ring topology defined by declaring the descending sequence of ideals
$\mathfrak{m}\supseteq\mathfrak{m}^{2}\supseteq...$
to be a base of neighbourhoods of $0$.
The {\bf residue field} of a local ring $A$,
which we usually denote by $k$,
is the quotient ring 
$A/\mathfrak{m}$,
and the natural quotient map
$$
    \res:A\longrightarrow k
$$
is called the {\bf residue map}.
The {\bf residue characteristic} of $A$ is by definition the characteristic of $k$.

For the sake of clarity, since maps between residue fields of local rings are of central importance in this paper, it will be suggestive to work with pairs $(A,k)$
consisting of a local ring $A$,
together with its residue field $k$.
Of course, such a pair is already determined by the local ring $A$,
and this notation fails to explicitly mention the maximal ideal or the residue map.
Without risk of confusion, we will also refer to such pairs as local rings.

\begin{lemma}[{Krull, \cite[Theorem 2]{Kru38}}]\label{lem:Noetherian}
Let $A$ be a Noetherian local ring.
Then $\bigcap_{n\in\mathbb{N}}\mathfrak{m}^{n}=\{0\}$.
In other words, $A$ is Hausdorff with respect to the local topology.
\end{lemma}

\begin{remark}[Other terminology]
Before we give our main definitions,
namely Definitions
\ref{def:preCohen.ring} and \ref{def:Cohen.ring},
we note that many closely related ideas have been named in the literature, both in original papers and textbooks.
Mac Lane, in \cite{Mac39c}, works with `{$p$-adic fields}' and `{$\mathfrak{p}$-adic fields}';
whereas Cohen, in \cite{Coh46}, prefers to work with `{local rings}' (which, for Cohen, are necessarily Noetherian), `{generalized local rings}', and `{$v$-rings}'.
Serre, in \cite[Chapter II, \S5]{Ser79}, defines a `{$p$-ring}' to be a ring $A$ which is Hausdorff and complete in the topology defined by a decreasing sequence
$\mathfrak{a}_{1}\supset\mathfrak{a}_{2}\supset...$
of ideals,
such that $\mathfrak{a}_{m}\mathfrak{a}_{n}\subseteq\mathfrak{a}_{m+n}$,
and for which $A/\mathfrak{a}_{1}$ is a perfect ring of characteristic $p$.
More recently, van den Dries, in {\cite[p.~132]{vdD14}}, defines a `{local $p$-ring}' to be a complete local ring $A$ with maximal ideal $pA$ and perfect residue field $A/pA$.

To minimise the risk of confusion with existing terminology,
we will not work with $v$-rings, $p$-adic fields, $\mathfrak{p}$-adic fields, $p$-rings, or local $p$-rings.
Instead, since Warner's point of view, in \cite[Chapter IX]{War93}, is closer to our own, it is his definition of `{Cohen ring}' that we adopt.
We hope the reader will forgive us for this, but we feel that none of the other notions (several of which are arguably more standard in the literature) exactly captures the right context for this paper.
\end{remark}

\begin{definition}\label{def:preCohen.ring}
A {\bf pre-Cohen ring} is a local ring $(A,k)$ such that
$A$ is Noetherian
and
the maximal ideal $\mathfrak{m}$ is $pA$.
\end{definition}

In particular, pre-Cohen rings are of residue characteristic $p$.
Turning to the question of the characteristic of $A$ itself, we note that a pre-Cohen ring need not even be an integral domain.
However, a pre-Cohen ring is either of characteristic $0$ or of characteristic $p^{m}$, for some $m\in\mathbb{N}_{>0}$.

\begin{lemma}\label{lem:strict}
For a pre-Cohen ring $(A,k)$, the following are equivalent:
\begin{enumerate}
\item
$A$ is of characteristic zero,
\item
$A$ is an integral domain,
\item
$A$ is a valuation ring.
\end{enumerate}
In this case, the corresponding valuation $v_{A}$ on the quotient field of $A$ is of mixed characteristic $(0,p)$, has value group isomorphic to $\mathbb{Z}$, with $v_{A}(p)$ minimum positive, and has residue field $k$.
\end{lemma}
\begin{proof}
This is a special case of {\cite[21.4 Theorem]{War93}}.
\end{proof}

\begin{definition}\label{def:strict}
If any (equivalently, all) of the conditions of Lemma \ref{lem:strict} are satisfied, then we say that $(A,k)$ is {\bf strict}.
\end{definition}

The word `strict' is borrowed from Serre,
\cite[II,\S5]{Ser79}.

\begin{remark}
In \cite{Coh46}, Cohen
writes in terms of {regular} Noetherian local rings.
A local ring is {\bf regular} if its Krull dimension is equal to the number of generators of its unique maximal ideal.
In the case of a pre-Cohen ring $(A,k)$,
the maximal ideal is by definition generated by one element, namely $p$.
Therefore, $(A,k)$ is regular if and only if its Krull dimension is $1$, which in turn holds if and only if $(A,k)$ is strict.
\end{remark}

A {\bf morphism}
of pre-Cohen rings,
which we write as
$\varphi:(A_{1},k_{1})\longrightarrow(A_{2},k_{2})$,
is a pair
$\varphi=(\varphi_{A},\varphi_{k})$
of 
ring homomorphisms
$\varphi_{A}:A_{1}\longrightarrow A_{2}$
and $\varphi_{k}:k_{1}\longrightarrow k_{2}$,
such that
\begin{enumerate}
\item
$\mathfrak{m}_{1}=\varphi_{A}^{-1}(\mathfrak{m}_{2})$,
i.e.~$\varphi_{A}$ is a morphism of local rings, and
\item
$\varphi_{k}\circ\res=\res\circ\varphi_{A}$.
\end{enumerate}
This is nothing more than a way of speaking about morphisms of local rings as pairs of maps, to match the pairs $(A,k)$.
Every morphism $\varphi_{A}$ of local rings
{induces}
a ring homomorphism
$\varphi_{k}:k_{1}\longrightarrow k_{2}$ such that
$(\varphi_{A},\varphi_{k})$ is a morphism of pre-Cohen rings.
From now on, by `morphism' we mean a morphism of pre-Cohen rings.
We will often (but not always) be concerned with morphisms $\varphi=(\varphi_{A},\varphi_{k})$
such that
$k_{2}/\varphi_{k}(k_{1})$ is separable.
By an {\bf embedding}, we mean a morphism $\varphi=(\varphi_{A},\varphi_{k})$
such that $\varphi_{A}$ is injective.
In the obvious way, we write $(A_{1},k_{1})\subseteq(A_{2},k_{2})$
if $A_{1}$ is a subring of $A_{2}$,
$k_{1}$ is a subfield of $k_{2}$,
and the inclusion maps form an embedding
$(A_{1},k_{1})\longrightarrow(A_{2},k_{2})$.

\begin{definition}[{Cf.~\cite[21.3 Definition]{War93}}]\label{def:Cohen.ring}
A pre-Cohen ring $(A,k)$ is a {\bf Cohen} ring if it is also complete,
i.e.\ complete with respect to the local topology.
\end{definition}

\begin{example}
$(\mathbb{Z}_{p},\mathbb{F}_{p})$
is a strict Cohen ring.
For each $m\in\mathbb{N}_{>0}$, $(\mathbb{Z}_{p}/p^{m}\mathbb{Z}_{p},\mathbb{F}_{p})$ is a non-strict Cohen ring
of characteristic $p^{m}$.
\end{example}

\begin{lemma}
Every pre-Cohen ring of positive characteristic is already a Cohen ring.
\end{lemma}
\begin{proof}
In a non-strict pre-Cohen ring the topology is discrete. Thus it is complete.
\end{proof}

Note that Cohen rings exist, for any residue field and any characteristic.
This foundational existence result goes back to the work of Hasse and Schmidt.

\begin{theorem}[{Existence Theorem}, {\cite[Theorem 20, p63]{HS34}}]\label{thm:HS}
Let $k$ be a field of characteristic $p$.
There exists a strict Cohen ring $(A,k)$.
Moreover, for each $m\in\mathbb{N}_{>0}$, there exists a Cohen ring $(A_{m},k)$ of characteristic $p^{m}$.
\end{theorem}

\begin{remark}[Inverse systems]\label{rem:inverse_system}
Let $(A,k)$ be a non-strict Cohen ring of characteristic $p^{m}$,
and let $n\in\mathbb{N}_{>0}$ such that $n\leq m$.
The image of $A$ under the quotient map
$\res_{m,n}:A\longrightarrow A/\mathfrak{m}^{n}$
is a non-strict Cohen ring of characteristic $p^{n}$,
again with residue field $k$.
Similarly, if $(A,k)$ is a strict Cohen ring,
and $n\in\mathbb{N}_{>0}$,
the image of $A$ under the quotient map
$r_{n}:A\longrightarrow A/\mathfrak{m}^{n}$
is a non-strict Cohen ring of characteristic $p^{n}$,
and as before with residue field $k$.
In both cases, the residue map of the quotient is induced by the residue map of $A$.
These quotient maps are compatible in the sense that,
for $n_{1}\leq n_{2}\leq m$,
we have $r_{n_{1}}=\res_{n_{2},n_{1}}\circ r_{n_{2}}$
and $\res_{m,n_{1}}=\res_{n_{2},n_{1}}\circ\res_{m,n_{2}}$.
In this way, given a strict Cohen ring $(A,k)$, we obtain an inverse system of non-strict Cohen rings $(A_{m},k)_{m>0}$ and maps $(\res_{m,n})_{n\leq m}$.
Conversely, beginning with an inverse system of non-strict Cohen rings $(A_{m},k)_{m>0}$,
where each $A_{m}$ has characteristic $p^{m}$,
and quotient maps $(\res_{m,n})_{n\leq m}$,
the inverse limit is a strict Cohen ring $(A,k)$, with quotient maps $(r_{m})_{m>0}$.
Clearly, these constructions are mutually inverse.

Furthermore, 
if $\varphi:(A,k)\longrightarrow(B,l)$ is an embedding of Cohen rings that are both of characteristic $p^{m}$
(resp., both strict),
then $\varphi$ induces embeddings
$\varphi_{n}:(A/\mathfrak{m}^{n},k)\longrightarrow(B/\mathfrak{m}^{n},l)$,
for all $n\leq m$
(resp., for all $n>0$).
On the other hand, if we begin with two inverse systems
$(A_{n},k)_{n>0}$ and $(B_{n},l)_{n>0}$,
such that $\mathrm{char}(A_{n})=\mathrm{char}(B_{n})=p^{n}$,
and a compatible family of embeddings 
$\varphi_{n}:(A_{n},k)\longrightarrow(B_{n},l)$,
for $n>0$,
there is a unique compatible embedding
$\varphi:(A,k)\longrightarrow(B,l)$,
where $(A,k)$ and $(B,l)$ are the corresponding inverse limits.
\end{remark}

\section{Representatives}
\label{section:representatives}
\subsection{Teichm\"{u}ller's multiplicative representatives}

The notion of `representatives' plays a key role in this subject.

\begin{definition}[Cf.~{\cite[\S 4.]{Tei36b}}]\label{def:Teichmuller}
Let $(A,k)$ be a pre-Cohen ring, and let $\alpha\in k$.
A {\bf representative} of $\alpha$
is some $a\in A$ with $\res(a)=\alpha$.
A {\bf multiplicative representative}
$a$ of $\alpha$
is a representative which is also
a $p^{n}$-th power in $A$, for all $n\in\mathbb{N}$.
A {\bf choice of representatives} is a partial function
\begin{align*}
s:k&\dasharrow A
\end{align*}
such that $s(\alpha)$ is a representative of $\alpha$.
To say that such a choice is {\bf for} $P$
means that $P$ is the domain of $s$, i.e.~$s:P\longrightarrow A$.
Obviously, such a map is a {\bf choice of multiplicative representatives} if
$s(\alpha)$ is a multiplicative representative of $\alpha$, for all $\alpha$ in the domain of $s$.
\end{definition}

We observe that, for any pre-Cohen ring $(A,k)$,
there exist many choices of representatives for $k$, and of course for any subset $P$ of $k$.
It is obvious that the largest subfield of $k$ for which multiplicative representatives may be chosen is the
perfect core
$k^{(p^{\infty})}$, which is by definition the subfield of elements which are $p^{n}$-th powers, for all $n\in\mathbb{N}$. Note that $k^{(p^{\infty})}$ is the largest perfect subfield of $k$.
For any subset $X$ of a ring, and any $n\in\mathbb{N}$, denote by $X^{(n)}=\{x^{n}\mid x\in X\}$ the set of $n$-th powers of elements of $X$.

The following straightforward lemma is the starting point for the study of multiplicative representatives.

\begin{lemma}[{\cite[Cf.~Hilfssatz 8]{Tei37}}]\label{lem:binomial}
Let $(A,k)$ be a pre-Cohen ring, let $a,b\in A$, and let $m,n\in\mathbb{N}$.
If
$a\equiv b\pmod{\mathfrak{m}^{m}}$,
then
$a^{p^{n}}\equiv b^{p^{n}}\pmod{\mathfrak{m}^{m+n}}$.
\end{lemma}

Perhaps the most important result about multiplicative representatives is
Theorem \ref{thm:representatives},
which is also due to Teichm\"{u}ller.

\begin{theorem}[Cf.~{\cite[\S 4.~Satz]{Tei36b}}]\label{thm:representatives}
Let $(A,k)$ be a Cohen ring.
There exists a unique choice of multiplicative representatives for $k^{(p^{\infty})}$:
\begin{align*}
	s:k^{(p^{\infty})}&\longrightarrow A.
\end{align*}
\end{theorem}

The proof can be found in many places,
for example
\cite[Lemma 7]{Coh46}.
In fact,
such a map $s$ is also multiplicative in a stronger sense, namely that
$s(\alpha)s(\beta)=s(\alpha\beta)$, for all $\alpha,\beta\in k^{(p^{\infty})}$,
cf.~\cite[Proposition 8(iii), \S4, Ch.~II]{Ser79}.

\begin{lemma}\label{lem:pm-representatives}
Let $(A,k)$ be a non-strict Cohen ring of characteristic $p^{m+1}$.
There is a unique choice of representatives
\begin{align*}
	s:k^{(p^{m})}\longrightarrow A^{(p^{m})}\subseteq A.
\end{align*}
\end{lemma}
\begin{proof}
Let $\alpha\in k$, let $a,b\in A$, and suppose that both $a^{p^{m}}$ and $b^{p^{m}}$ are representatives of $\alpha^{p^{m}}$.
Then both $a$ and $b$ are representatives of $\alpha$,
and in particular $a\equiv b\pmod{\mathfrak{m}}$.
By Lemma~\ref{lem:binomial},
$a^{p^{m}}\equiv b^{p^{m}}\pmod{\mathfrak{m}^{m+1}}$.
Since the characteristic of $A$ is $p^{m+1}$,
we have
$a^{p^{m}}=b^{p^{m}}$.
\end{proof}

The unique choice of representatives in Lemma~\ref{lem:pm-representatives} will be called the {\bf $p^{m}$-representatives}.

\subsection{\texorpdfstring{$\lambda$}{λ}-maps}
\label{section:lamba.maps}

A subset $\bfbeta\subseteq k$ is {\bf $p$-independent} over a subfield $C\subseteq k$ if
$[k^{(p)}C(\beta_{1},...,\beta_{r}):k^{(p)}C]=p^{r}$,
for all pairwise distinct elements
$\beta_{1},\dots,\beta_{r}\in\bfbeta$,
and for all $r\in\mathbb{N}$;
and $\bfbeta$ is a {\bf $p$-basis} over $C$ if furthermore
$k=k^{(p)}C(\bfbeta)$.
Equivalently, a $p$-basis is a maximal $p$-independent subset.
If $C$ is the prime field $\mathbb{F}_{p}$, then we just speak of `$p$-independence' and `$p$-bases'.
The cardinality of a $p$-basis of $k$ does not depend on the choice of any particular $p$-basis, and it is called the {\bf imperfection degree}\footnote{Imperfection degree is sometimes called {\em Ershov degree} or {\em $p$-degree}.} of $k$.
See \cite{Tei36a},
\cite{Mac39a}, and
\cite{Mac39b},
for more information on $p$-independence and $p$-bases.

Our next task is to develop the theory of $\lambda$-maps and $\lambda$-representatives with respect to arbitrary $p$-independent subsets $\bfbeta$, which certainly may be infinite, since in general the imperfection degree of a field may be any cardinal number.

For a cardinal $\nu$, and $n\in\mathbb{N}$,
we define
\begin{align*}
    P_{\nu,n}&:=\Big\{(i_{\mu})_{\mu<\nu}\;\Big|\;\mbox{$|\{\mu<\nu\;|\;i_{\mu}\neq0\}|<\infty$ and $\forall \mu<\nu,\;0\leq i_{\mu}<p^{n}$ }\Big\}
\end{align*}
to be the set of the multi-indices of finite support, in $\nu$-many elements,
and
in which each index is a non-negative integer strictly less that $p^{n}$.
In this context, `finite support' means that any such multi-index contains only finitely many non-zero indices.
We emphasise that this set is just a technical device to facilitate our analysis of $p$-independence,
and we will be mostly interested in the case $n=1$.
For
a $p$-independent set
$\bfbeta=(\beta_{\mu})_{\mu<\nu}$,
and $I=(i_{\mu})_{\mu<\nu}\in P_{\nu,n}$,
we write
\begin{align*}
    \beta^{I}&:=\prod_{\mu<\nu}\beta_{\mu}^{i_{\mu}}
\end{align*}
for the $I$-th monomial of $\bfbeta$.
The set $\{\beta^{I}\mid I\in P_{\mu,n}\}$ is a $k^{(p^{n})}$-linear basis of $k^{(p^{n})}(\bfbeta)$.
Therefore,
for each $\alpha\in k^{(p^{n})}(\bfbeta)$,
there is a unique family
$(\lambda^{I}_{\bfbeta}(\alpha))_{I\in P_{\nu,n}}$
of elements of $k$
such that
\begin{align}\label{eq:lambda}\tag{3.2.1}
    \alpha&=\sum_{I\in P_{\nu,n}}\beta^{I}\lambda^{I}_{\bfbeta}(\alpha)^{p^{n}}.
\end{align}
Note that this sum is finite since $\lambda^{I}_{\bfbeta}(\alpha)$ is zero for cofinitely many $I\in P_{\nu,n}$.
We refer to
\begin{align*}
    \lambda^{I}_{\bfbeta}:k^{(p^{n})}(\bfbeta)&\longrightarrow k\\
    \alpha&\longmapsto\lambda^{I}_{\bfbeta}(\alpha)
\end{align*}
as the $I$-th {\bf $\lambda$-map} with respect to $\bfbeta$.

\subsection{Representatives and subrings}

We prove in Proposition~\ref{prp:generation} that each Cohen ring of characteristic $p^{m+1}$ is generated by the union of the $p^{m}$-representatives and a set of representatives for a $p$-basis of its residue field; this is a key ingredient of
Theorem~\ref{thm:MIT}.
We begin with a lemma.

\begin{lemma}\label{lem:expansion}
Let $(A,k)$ be a non-strict Cohen ring of characteristic $p^{m+1}$.
Let $s:k\longrightarrow A$ be a choice of representatives.
Then every element of $A$ admits a unique representation as a sum $\sum_{i=0}^{m}s(\alpha_{i})p^{i}$, where $\alpha_{i}\in k$.
In particular, $A$ is generated as a ring by the image $s(k)$ of $s$.
\end{lemma}
\begin{proof}
We claim that, for all $n\in\mathbb{N}$ and for all $a\in A$, there exist $\alpha_{0},\ldots,\alpha_{n}\in k$
such that $a-\sum_{i=0}^{n}s(\alpha_{i})p^{i}\in\mathfrak{m}^{n+1}_{A}$.
We prove this by induction on $n$.
The base case $n=0$ follows from the fact that $s$ is a choice of representatives for $k$.
Suppose the statement holds for some $n\in\mathbb{N}$.
Let $a\in A$ and denote $\alpha:=\res(a)\in k$.
Let $\hat{a}\in A$ be such that $a=s(\alpha)+p\hat{a}$.
By the inductive hypothesis, there exist $\alpha_{0},\ldots,\alpha_{n}\in k$ such that $\hat{a}-\sum_{i=0}^{n}s(\alpha_{i})p^{i}\in\mathfrak{m}^{n+1}$.
Denote $b:=s(\alpha)+\sum_{i=0}^{n}s(\alpha_{i})p^{i+1}$.
Now $a-b=p\hat{a}-p\sum_{i=0}^{n}s(\alpha_{i})p^{i}\in\mathfrak{m}^{n+2}$.

The uniqueness part of the statement follows from the fact that multiplication by $p^{i}$ is an isomorphism between the additive groups $A/pA$ and $p^{i}A/p^{i+1}A$, for $i\leq m$.
If the characteristic of $A$ is $p^{m+1}$, then as soon as $m+1\leq n+1$, the claim shows that $A$ is generated by $s(k)$.
\end{proof}

\begin{proposition}\label{prp:generation}
Let $(A,k)$ be a non-strict Cohen ring characteristic $p^{m+1}$,
and let $\bfbeta$ be a $p$-basis of $k$ with representatives $s:\bfbeta\longrightarrow A$.
Denote by $s_{A}:k^{(p^{m})}\longrightarrow A^{(p^{m})}$ the choice of $p^{m}$-representatives from Lemma~\ref{lem:pm-representatives}.
Then $A$ is generated by $s(\bfbeta)\cup s_{A}(k^{(p^{m})})$.
\end{proposition}
\begin{proof}
Let $B:=[s(\bfbeta)\cup s_{A}(k^{(p^{m})})]$ be the subring of $A$ generated by the images of $s$ and $s_{A}$.
By Lemma~\ref{lem:expansion}, it suffices to show that $B$ contains a set of representatives for each element of $k$.
For  $I=(i_{\mu})_{\mu<\nu}\in P_{\nu,n}$,
we write
\begin{align*}
    s(\beta^{I})&:=\prod_{\mu<\nu}(s(\beta_{\mu}))^{i_{\mu}}.
\end{align*}
We now define a function
$S:k\longrightarrow A$
by setting 
\begin{align*}
	S(\alpha)&:=\sum_{I\in P_{\nu,m}}s(\beta^{I})s_{A}(\lambda^{I}_{\bfbeta}(\alpha)^{p^{m}}),
\end{align*}
for $\alpha \in k$.
Applying the residue map, and comparing with Equation~(\ref{eq:lambda}),
we can see that $S(\alpha)$ is a representative of $\alpha$.
Moreover, $S(\alpha)$ is an element of $B$, as required.
\end{proof}

\begin{remark}[Inverse systems with representatives]\label{rem:inverse_system_2}
Let $k$ be a field of characteristic $p$.
Let $(A_{m},k)_{m\in\mathbb{N}}$ be an inverse system of Cohen rings,
in the sense of Remark~\ref{rem:inverse_system}, where each $A_{m}$ has characteristic $p^{m+1}$,
and denote by $(\res_{n,m})_{n\geq m}$ the transition maps.
Let $\bfbeta$ be a $p$-basis of $k$ and, for each $m$, let $s_{m}:\bfbeta\longrightarrow A_{m}$ be a choice of representatives.
Suppose that the maps $s_{m}$ are compatible in the sense that $\res_{n,m}\circ s_{n}=s_{m}$, for $n\geq m$.
Then $(A_{m},k,s_{m})_{m\in\mathbb{N}}$ forms an inverse system.

Let $(A,k)$ be the inverse limit of $(A_{m},k)_{m\in\mathbb{N}}$ with projections $(r_{m})_{m\in\mathbb{N}}$.
By Remark~\ref{rem:inverse_system}, $(A,k)$ is a strict Cohen ring.
The inverse limit of the maps $s_{m}$ is a choice of representatives $s:\bfbeta\longrightarrow A$.
\end{remark}

\section{The Teichm\"{u}ller Embedding Process}

At the heart of all the structural arguments about Cohen rings is Teichm\"{u}ller's embedding process, which we discuss in this section.
The original formulation can be found in \cite[\S 7]{Tei36b}.
Indeed,
Mac Lane attributes this technique to Teichm\"{u}ller,
and describes it as the `Teichm\"{u}ller embedding process'.
See {\cite[Theorem 6]{Mac39c}} for Mac Lane's version.
In \cite[Lemma 12]{Coh46}, Cohen rewrote Teichm\"{u}ller's embedding process for an arbitrary complete local ring.

\begin{theorem}[{Teichm\"{u}ller Embedding Process}]\label{thm:TEP}
Let $(A,k)$ be a Cohen ring,
let $\bfbeta\subseteq k$ be $p$-independent
with representatives $s:\bfbeta\longrightarrow A$.
There exists a Cohen ring $(A^{T},k^{T})\supseteq(A,k)$ such that
\begin{enumerate}
\item $k^{T}=k(\bfbeta^{(p^{-\infty})})$,
where $\bfbeta^{(p^{-\infty})}=\{\beta^{p^{-n}}\mid\beta\in\bfbeta,n\in\mathbb{N}\}$,
\item $s$ coincides with the restriction to $\bfbeta$ of the unique choice of multiplicative representatives
$(k^{T})^{(p^{\infty})}\longrightarrow A^{T}$.
\end{enumerate}
\end{theorem}
\begin{proof}
This proof is closely based on those of Teichm\"{u}ller (\cite[\S7]{Tei36b}) and Cohen (\cite[Lemma 12]{Coh46}).
It is a recursive construction.
We begin by formally adjoining a $p$-th root of each $s(\beta)$, for each $\beta\in\bfbeta$.
More constructively, we introduce a family of new variables $(X_{\beta}:\beta\in\bfbeta)$, and let
$$
    A':=A[X_{\beta}:\beta\in\bfbeta]/\big(X_{\beta}^{p}-s(\beta):\beta\in\bfbeta\big).
$$
That is, $A'$ is the quotient of the ring 
$A[X_{\beta}:\beta\in\bfbeta]$ 
by the ideal generated by the polynomials $X_{\beta}^{p}-s(\beta)$, for $\beta\in\bfbeta$.
The natural map $A\longrightarrow A'$ is injective, and we identify $A$ with its image in $A'$.

Taking the quotient of $A'$ by $pA'$ yields the field
$k':=k(\beta^{p^{-1}}:\beta\in\bfbeta)$,
and so $pA'$ is maximal.
Indeed, since $A$ is local with unique maximal ideal $pA$, the maximal ideals of $A'$ are those lying over $pA$,
which shows that $pA'$ is the unique maximal ideal of $A'$.
Thus $(A',k')$ is a pre-Cohen ring,
and we have $(A,k)\subseteq (A',k')$.
Note that $\bfbeta^{(p^{-1})}:=\{\beta^{p^{-1}}\mid\beta\in\bfbeta\}$
is $p$-independent in $k'$.
Indeed, for each $\beta\in\bfbeta$, we write $s'(\beta^{p^{-1}})$ for the image of $X_{\beta}$ in the quotient ring $A'$.
Then $s':\bfbeta^{(p^{-1})}\longrightarrow A'$ is a choice of representatives,
and
$$
    s'(\beta^{p^{-1}})^{p}=s(\beta),
$$
for all $\beta\in\bfbeta$.

Beginning with $(A,k)$, we 
continue this process recursively, with recursive step
$(A,k)\longmapsto(A',k')$.
In this way, we construct a chain
$(A_{n},k_{n})_{n\in\mathbb{N}}$
of pre-Cohen rings,
such that $\bfbeta^{(p^{-n})}:=\{\beta^{p^{-n}}\mid\beta\in\bfbeta\}$ is $p$-independent in $k_{n}=k(\bfbeta^{(p^{-n})})$ and 
$s_{n}:\bfbeta^{(p^{-n})}\longrightarrow A_{n}$ is a choice of representatives, such that
\begin{align*}
    s_{n}(\beta^{p^{-n}})^{p^{n}}&=s(\beta),
\end{align*}
for all $n\in\mathbb{N}$ and all $\beta\in\bfbeta$.

The morphisms in this chain are embeddings, which we may even view as inclusions, by identifying of each $(A_{n},k_{n})$ with its image in $(A_{n+1},k_{n+1})$.
The direct limit is a pre-Cohen ring $(A_{\infty},k_{\infty})\supseteq(A,k)$.
Taking the completion, we obtain a Cohen ring $(A^{T},k^{T})\supseteq(A,k)$.
The union $s^{T}:=\bigcup_{n}s_{n}$ is a choice of representatives for
$\bfbeta^{T}:=\bigcup_{n}\bfbeta^{(p^{-n})}$
which commutes with the Frobenius map.
By construction, we have $k^{T}=k(\bfbeta^{(p^{-\infty})})$,
and so $\bfbeta^{T}\subseteq (k^{T})^{(p^{\infty})}$.
Therefore $s^{T}$ coincides with the restriction to $\bfbeta^{T}$ of the unique choice of multiplicative representatives $(k^{T})^{(p^{\infty})}\longrightarrow A^{T}$,
as required.
\end{proof}

\section{Mac Lane's Identity Theorem}

In this section we consider Cohen subrings of Cohen rings.
We study the `identity' of such subrings inside their overrings:
in Theorem \ref{thm:MIT},
which was first clearly articulated by Mac Lane,
we show that such a subring is determined by a choice of representatives of a $p$-basis of its residue field.

Teichm\"{u}ller's discussion of this issue can be found in \cite[\S8]{Tei36b}.
Developing these ideas, Mac Lane's theorems
\cite[Theorem 7]{Mac39c}
and
\cite[Theorem 12]{Mac39c}
show that a complete subfield of a $\mathfrak{p}$-adic field, in his language, is determined by a choice of representatives for a $p$-basis of the residue field.
Indeed, in our view, Mac Lane is the first to have clearly articulated this portion of the overall argument.
Nevertheless, we closely follow Cohen's exposition, particularly relevant parts of his proof of \cite[Theorem 11]{Coh46},
which is in fact the theorem we will discuss in the next section.

\begin{theorem}[{Mac Lane's Identity Theorem}]\label{thm:MIT}
Let $(B,l)$ be a pre-Cohen ring, 
with pre-Cohen subrings $(A_{1},k_{1})$ and $(A_{2},k_{2})$.
Suppose that
$(A_{2},k_{2})$ is a Cohen ring,
i.e.~is complete,
and that
$k_{1}\subseteq k_{2}$.
Let $\bfbeta$ be a $p$-basis of $k_{1}$
with representatives
$s:\bfbeta\longrightarrow A_{1}$.
If $s(\bfbeta)\subseteq A_{2}$
then $(A_{1},k_{1})\subseteq (A_{2},k_{2})$.
\end{theorem}
\begin{proof}
First we work in the case that $(B,l)$ is non-strict of characteristic $p^{m+1}$.
For $i\in\{1,2\}$,
let $s_{i}:k_{i}^{(p^{m})}\longrightarrow A_{i}^{(p^{m})}\subseteq A_{i}$, and let $s_{B}:l^{(p^{m})}\longrightarrow B^{(p^{m})}\subseteq B$,
be the unique choices of $p^{m}$-representatives from Lemma~\ref{lem:pm-representatives}.
Let $\alpha\in k_{i}^{(p^{m})}$.
By Lemma~\ref{lem:pm-representatives}, $s_{B}(\alpha)$ is the unique element of $B^{(p^{m})}$ with residue $\alpha$; but $s_{i}(\alpha)$ is another such element that happens to lie in $A_{i}^{(p^{m})}\subseteq B^{(p^{m})}$.
Therefore $s_{i}(\alpha)=s_{B}(\alpha)$,
which means that $s_{B}$ extends $s_{i}$.

Since $k_{1}\subseteq k_{2}\subseteq l$, we have $k_{1}^{(p^{m})}\subseteq k_{2}^{(p^{m})}\subseteq l^{(p^{m})}$.
It follows that $s_{2}$ extends $s_{1}$.
In particular, the image $s_{1}(k_{1}^{(p^{m})})$ of $s_{1}$ is contained in the image of $s_{2}$,
which in turn is contained in $A_{2}$.
Also note that $s(\bfbeta)\subseteq A_{2}$, by assumption.
By Proposition~\ref{prp:generation}, the subring of $B$ generated by $s_{1}(k_{1}^{(p^{m})})\cup s(\bfbeta)$ is $A_{1}$.
Therefore $A_{1}\subseteq A_{2}$.

Finally, we suppose that $(B,l)$ is strict.
For all $m\in\mathbb{N}$,
by the preceding paragraph,
the residue ring of $(A_{1},k_{1})$ of characteristic $p^{m+1}$
is a subring of the corresponding residue ring of $(A_{2},k_{2})$.
Since $(A_2,k_2)$ is complete by assumption,
Remark \ref{rem:inverse_system} implies that
$(A_{1},k_{1})$ is a subring of $(A_{2},k_{2})$.
\end{proof}

\section{Cohen's Homomorphism Theorem and Structure Theorem}
\label{section:structure}

The remaining ingredient of a structure theorem is the relationship between two arbitrary Cohen rings with the same residue field.
Such a relationship exists, in the form of a morphism, and such a morphism is uniquely determined by specifying the image of a set of representatives of a $p$-basis of the residue field.

Cohen's paper \cite{Coh46} appears to be the first to study the case of characteristic $p^{m}$, $m>0$.
In this section we state and prove a version of Cohen's Theorem, 
\cite[Theorem 11]{Coh46}, suitable for our setting.

\begin{definition}\label{def:lift}
Let $(A,k)$ and $(B,l)$ be pre-Cohen rings,
and let $\varphi=(\varphi_{A},\varphi_{k}):(A,k)\longrightarrow(B,l)$ be a morphism.
Also, let $\bfbeta\subseteq k$ be a $p$-basis of $k$,
and let $s_{A}:\bfbeta\longrightarrow A$ and $s_{B}:\varphi_{k}(\bfbeta)\longrightarrow B$
be representatives.
We say that $\varphi$
{\bf respects}
$s_{A}$ and $s_{B}$ if
$\varphi_{A}\circ s_{A}=s_{B}\circ\varphi_{k}|_{\bfbeta}$.
\end{definition}

\begin{figure}[ht]
$$
\renewcommand{\labelstyle}{\textstyle}
    \xymatrix@=3em{
        B\ar@{.>}[rrr]^{\res_{B}}
        &&&
        l
        \ar@{:>}@/^2pc/[lll]^(0.50){s_{B}}
        \\
        &&&
        \\
        A\ar@{.>}[rrr]^{\res_{A}}
        \ar@{->}[uu]_{}^{\varphi_{A}}
        &&&
        k\ar@{->}[uu]_{}^{\varphi_{k}}
        \ar@{:>}@/^2pc/[lll]^(0.50){s_{A}}
    }
$$
\caption{Illustration of Definition \ref{def:lift}}
\label{fig:3}
\end{figure}
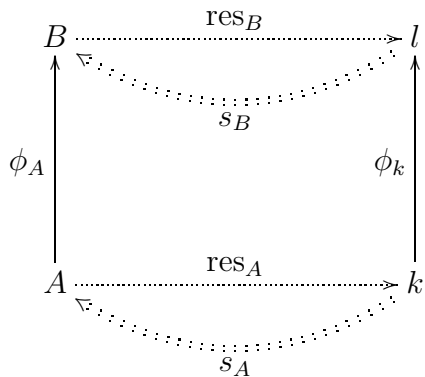

\begin{theorem}[{Cohen's Homomorphism Theorem}]\label{thm:CoHo}
Let $(A,k)$ and $(B,l)$ be Cohen rings,
and let $\varphi_{k}:k\longrightarrow l$ be an embedding of fields such that $l/\varphi_{k}(k)$ is a separable extension.
Let $\bfbeta$ be a $p$-basis of $k$ and let $s_{A}:\bfbeta\longrightarrow A$ and $s_{B}:\varphi_{k}(\bfbeta)\longrightarrow B$ be representatives.
Suppose that $(A,k)$ is strict or that $(A,k)$ is non-strict of characteristic at least
that of $(B,l)$.
Then there exists a unique ring homomorphism $\varphi_{A}:A\longrightarrow B$ such that
$$
    \varphi=(\varphi_{A},\varphi_{k}):(A,k)\longrightarrow(B,l)
$$
is a morphism which respects $s_{A}$ and $s_{B}$.

Moreover,
if $(A,k)$ and $(B,l)$ have the same characteristic, then $\varphi$ is an embedding.
\end{theorem}
\begin{proof}
This proof is a construction closely 
based on that of Cohen (\cite[Theorem 11]{Coh46}).
For notational simplicity, we identify $k$ with its image in $l$ under the embedding $\varphi_{k}$.
Then $\varphi_{k}$ is the inclusion map $\mathrm{id}$, and $l/k$ is a separable extension.
It suffices to construct a ring homomorphism
$A\longrightarrow B$, which induces the inclusion map $k\longrightarrow l$.
Throughout, we denote the residue maps by $\mathrm{res}_A: A \longrightarrow k$
and $\mathrm{res}_B: B \longrightarrow l$  respectively.

To begin with, we suppose that $(A,k)$ is strict and that $k$ is perfect. 
Thus $\bfbeta$ is empty, and we dispense with both of the maps $s_{A}$ and $s_{B}$.
Since $(A,k)$ is a strict Cohen ring, we have
$(\mathbb{Z}_{p},\mathbb{F}_{p})\subseteq(A,k)$,
and there is the following natural ring homomorphism:
\begin{align*}
    \varphi_{0}:\mathbb{Z}_{p}\longrightarrow B.
\end{align*}
Let $T$ be a transcendence basis of $k/\mathbb{F}_{p}$.
Since $k$ is perfect, we have
$T\subseteq k^{(p^{\infty})}=k\subseteq l^{(p^{\infty})}\subseteq l$.
By Theorem \ref{thm:representatives},
there are unique choices of multiplicative representatives:
$s_{A,0}:k\longrightarrow A$
and
$s_{B,0}:l^{(p^{\infty})}\longrightarrow B$.
Note that $s_{A,0}(T)$ is algebraically independent over $\mathbb{Z}_{p}$
since $T$ is algebraically independent over $\mathbb{F}_{p}$.
We consider the subring $\mathbb{Z}_{p}[s_{A,0}(T)]\subseteq A$,
with $\res_{A}(\mathbb{Z}_{p}[s_{A,0}(T)])=\mathbb{F}_{p}[T]$.
We may extend $\varphi_{0}$ to a ring homomorphism
\begin{align*}
    \varphi_{1,0}:\mathbb{Z}_{p}[s_{A,0}(T)]\longrightarrow B
\end{align*}
by declaring $\varphi_{1,0}(s_{A,0}(t))=s_{B,0}(t)$, for each $t\in T$.
In fact, for each $n\in\mathbb{N}$,
we consider the subring $\mathbb{Z}_{p}[s_{A,0}(T^{(p^{-n})})]\subseteq A$,
with $\res_A(\mathbb{Z}_{p}[s_{A,0}(T^{(p^{-n})})])=\mathbb{F}_{p}[T^{(p^{-n})}]$.
As in the case $n=0$, we construct a ring homomorphism
\begin{align*}
    \varphi_{1,n}:\mathbb{Z}_{p}[s_{A,0}(T^{(p^{-n})})]\longrightarrow B,
\end{align*}
by declaring $\varphi_{1,n}(s_{A,0}(t^{p^{-n}}))=s_{B,0}(t^{p^{-n}})$, for each $t\in T$.
Since $s_{A,0}$ and $s_{B,0}$ are multiplicative,
the family $(\varphi_{1,n})_{n\in\mathbb{N}}$ of ring homomorphisms forms an increasing chain.
Taking the direct limit (i.e.\ union), we obtain the subring
$A_{0}:=\mathbb{Z}_{p}[s_{A,0}(T^{(p^{-n})})\;|\;n\in\mathbb{N}]\subseteq A$,
with $\res_{A}(A_{0})=\mathbb{F}_{p}[T^{(p^{-n})}\mid n\in\mathbb{N}]$,
and we obtain the ring homomorphism
\begin{align*}
    \varphi_{2}:A_{0}\longrightarrow B.
\end{align*}
Localising $A_{0}$ at $A_{0}\cap pA$,
we obtain the local ring $A_{1}:=(A_{0})_{A_{0}\cap pA}\subseteq A$,
with $\res_{A}(A_{1})=\mathbb{F}_{p}(T)^{\mathrm{perf}}$,
and we extend $\varphi_{2}$ to a ring homomorphism
\begin{align*}
    \varphi_{3}:A_{1}\longrightarrow B.
\end{align*}

The final part of this construction is to extend $\varphi_{3}$ to have domain $A$.
Since strict Cohen rings are henselian valuation rings, and $k/\mathbb{F}_{p}(T)^{\mathrm{perf}}$ is separable algebraic,
this prolongation can be accomplished by a direct application of Hensel's Lemma, as in e.g.~\cite[Lemma 9.30]{Kuhlmann2011}.
More precisely, for a separable irreducible polynomial $f\in A_{1}[X]$ and $\alpha\in k$ with $\mathrm{res}_{A}(f)(\alpha)=0$, by Hensel's Lemma we obtain $a\in A$ such that $f(a)=0$.
Likewise, we obtain $b\in B$ with $\varphi_{3}(f)(b)=0$.
We now extend $\varphi_{3}$ to a morphism
\begin{align*}
    \varphi_{4}:A_{1}[a]&\longrightarrow B
\end{align*}
by sending $a\longmapsto b$.
Note that $\res_{A}(A_{1}[a])=\mathbb{F}_{p}(T)^{\mathrm{perf}}(\alpha)$.
Taking the direct limit of ring homomorphisms constructed in this way,
we obtain a ring homomorphism
\begin{align*}
    \varphi:A&\longrightarrow B
\end{align*}
that induces the inclusion map on the residue fields
and respects $s_{A}$ and $s_{B}$,
as required.
It remains to show that $\varphi$ is the unique such morphism, but this follows from Theorem \ref{thm:MIT}, applied to the case that $\bfbeta$ is empty.

Remaining under the assumption that $(A,k)$ is strict, we turn to the case that $k$ is imperfect.
We are given a $p$-basis
$\bfbeta$ of $k$
with 
representatives
$s_{A}:\bfbeta\longrightarrow A$
and
$s_{B}:\bfbeta\longrightarrow B$.
Note that $\bfbeta$ is $p$-independent in $l$, by our assumption that $l/k$ is separable.
By Theorem \ref{thm:TEP},
there exists a Cohen ring
$(A^{T},k^{T})\supseteq(A,k)$
such that
\begin{enumerate}
\item $k^{T}=k^{\mathrm{perf}}$, and
\item $s_{A}$ is the restriction to $\bfbeta$ of the multiplicative representatives $s_{A}^{T}:k^{T}\longrightarrow A^{T}$.
\end{enumerate}
By another application of Theorem \ref{thm:TEP},
there exists a Cohen ring
$(B^{T},l^{T})\supseteq(B,l)$
such that
\begin{enumerate}
\setcounter{enumi}{2}
\item $l^{T}=l(\bfbeta^{(p^{-\infty})})$, and
\item $s_{B}$ is the restriction to $\bfbeta$ of the multiplicative representatives $s_{B}^{T}:(l^{T})^{(p^{\infty})}\longrightarrow B^{T}$.
\end{enumerate}
Since $k^{T}\subseteq l^{T}$ and $k^{T}$ is perfect, by the first part of this proof there exists a unique morphism
\begin{align*}
    \varphi:A^{T}\longrightarrow B^{T}
\end{align*}
that induces the inclusion map on the residue fields.

By {\bf(ii)}, the composition $\varphi\circ s_{A}:\bfbeta\longrightarrow\varphi(A^{T})$ coincides with the unique choice of multiplicative representatives for $\bfbeta$ in $\varphi(A^{T})$;
and by {\bf(iv)}, also $s_{B}$ coincides with the unique choice of multiplicative representatives for $\bfbeta$ in $B^{T}$.
Applying Theorem~\ref{thm:representatives},
and since $\varphi((A^{T})^{(p^{\infty})})\subseteq(B^{T})^{(p^{\infty})}$,
we have $\varphi\circ s_{A}=s_{B}$.
If follows that both subrings $\varphi(A)$ and $B$ of $B^{T}$ contain $\varphi(s_{A}(\bfbeta))=s_{B}(\bfbeta)$.
Since also $k\subseteq l$, we may apply Theorem~\ref{thm:MIT} to deduce that $\varphi(A)\subseteq B$.
Therefore
$\varphi$ restricts to a ring homomorphism $A\longrightarrow B$ that induces the inclusion map on the residue fields
and respects $s_{A}$ and $s_{B}$.
The uniqueness of $\varphi$ again follows from Theorem \ref{thm:MIT}.
Note that, the non-trivial proper ideals of $A$ are $\mathfrak{m}_{C}^{n}$, for $n>0$, and the quotient $A/\mathfrak{m}_{C}^{n}$ has characteristic $p^{n}$.
Thus the morphism we have constructed is an embedding if and only if $(B,l)$ is also strict.  
This completes the proof in the case that $(A,k)$ is strict.

Finally, we turn to the case that $(A,k)$ and $(B,l)$ are non-strict, of characteristics $p^{n}$ and $p^{m}$, respectively, for $n\geq m$.
By Theorem~\ref{thm:HS}, there exists a strict Cohen ring $(C,k)$.
Choose representatives
$s_{C}:\bfbeta\longrightarrow C$.
By the first part of this Theorem, there is a unique ring homomorphism
$\varphi:C\longrightarrow A$
which induces the identity map on the residue field and respects $s_{C}$ and $s_{A}$.
Again note that the non-trivial proper ideals in $(C,k)$ are $\mathfrak{m}_{C}^{n}$, for $n\in\mathbb{N}_{>0}$.
Therefore $\varphi$ is the composition of the quotient $C\longrightarrow C/\mathfrak{m}_{C}^{n}$ with an isomorphism $C/\mathfrak{m}_{C}^{n}\longrightarrow A$.
Likewise there is a unique ring homomorphism
$\psi:C\longrightarrow B$
which induces the inclusion map on the residue fields and respects $s_{C}$ and $s_{B}$.
Again $\psi$ is the composition of the quotient of $C\longrightarrow C/\mathfrak{m}_{C}^{m}$ with an embedding $C/\mathfrak{m}_{C}^{m}\longrightarrow B$.
Since $n\geq m$, $\mathfrak{m}_{C}^{n}\subseteq\mathfrak{m}_{C}^{m}$, and thus these homomorphisms give rise to a ring homomorphism $A\longrightarrow B$ that induces the inclusion map on the residue fields and respects $s_{A}$ and $s_{B}$.
Once again, the uniqueness follows from Theorem~\ref{thm:MIT}.
Note that the morphism is an embedding if and only if $m=n$.
\end{proof}

\begin{remark}\label{rem:quotient}
In the setting of Theorem~\ref{thm:CoHo}, and in the case that $(A,k)$ is strict and $(B,l)$ is of characteristic $p^{m}$,
the resulting morphism $\varphi$ factors into a composition of the natural quotient map
$(A,k)\longrightarrow(A/\mathfrak{m}_{A}^{m},k)$
and an embedding
$(A/\mathfrak{m}_{A}^{m},k)\longrightarrow(B,l)$.
\end{remark}

In our applications of Cohen's Homomorphism Theorem in the second part of the paper, 
we require the following consequence:

\begin{corollary}[Relative Embedding Theorem]\label{cor:CoHo_relative}
Let $(A_{1},k_{1})$ and $(A_{2},k_{2})$ be two Cohen rings,
and let $(A_{0},k_{0})$ be a Cohen subring common to both.
Assume we are given an embedding of residue fields $\varphi_k:k_1 \longrightarrow k_2$ over $k_0$ and that both $k_1/k_0$ and $k_2/\varphi_k(k_1)$ are separable. 
Then, there is an embedding $\varphi$ of $(A_{1},k_{1})$ into $(A_{2},k_{2})$ which induces
$\varphi_k$ and fixes
$A_{0}$ pointwise.
Moreover, if $\varphi_{k}$ is an isomorphism then $\varphi$ is an isomorphism.
\end{corollary}
\begin{proof}
Note that by assumption, $(A_1,k_1), (A_2,k_2)$ and $(A_0,k_0)$ have the same
characteristic.
Let $\bfbeta_0$ be a $p$-basis of $k_0$, and $s_0:\bfbeta_0 \longrightarrow A_0$ be a choice of representatives.
We first show the existence of an embedding of Cohen rings 
$\varphi:A_1 \longrightarrow A_2$ which induces $\varphi_k$ and fixes $s_0(\bfbeta_{0})$.
Since $k_1/k_0$ is separable, we can find a $p$-basis $\bfbeta_1$ of $k_{1}$ prolonging $\bfbeta_0$,
and a choice of representatives $s_1:\bfbeta_1 \longrightarrow A_1$ prolonging $s_0$.
Note that since $\varphi_k$ restricts to the identity on $k_0$,
$\varphi_{k}(\bfbeta_1)$ also contains $\bfbeta_0$.
We now choose representatives
$s_2:\varphi_{k}(\bfbeta_1) \longrightarrow A_2$ such that $s_2$ prolongs $s_0$.
By Theorem \ref{thm:CoHo},
and since $k_{2}/\varphi_{k}(k_{1})$ is separable,
there is a morphism $\varphi:A_1 \longrightarrow A_2$
that respects $s_1$ and $s_2$ (and hence fixes $s_0(\bfbeta_{0})$) and induces
$\varphi_k$.
Since $A_{1}$ and $A_{2}$ have the same characteristic, $\varphi$ is an embedding.

Now, let $\varphi:A_1\longrightarrow A_2$ be any embedding which induces $\varphi_k$ and
fixes $s_0(\bfbeta_{0})$. Then the restriction $\varphi_0$ of $\varphi$ to $A_0$ is a ring
isomorphism between $A_0$ and $\varphi_0(A_0)$. 
Since $s_0(\bfbeta_0)$ is contained in $\varphi_0(A_0)$, by Mac Lane's Identity Theorem
(Theorem~\ref{thm:MIT}) we get $A_0 \subseteq \varphi_0(A_0)$.
Symmetrically, as
$s_0$ is also a choice of representatives for a $p$-basis of the residue field of $\varphi_0(A_0)$, we get $\varphi_0(A_0) \subseteq A_0$.
Therefore $\varphi_{0}$ is an automorphism of $A_{0}$.
By assumption, $\varphi_k$ restricts to the identity on $k_0$, and so $\varphi$ induces the identity on the residue field of $A_{0}$.
By construction of $\varphi$, $\varphi_{0}$ fixes $s_0(\bfbeta_{0})$.
Hence in particular $\varphi_{0}$ respects $s_{0}$ and $s_{0}$
(note that $s_{0}$ is a choice of representatives of the domain and codomain of $\varphi_{0}$, which are both $A_{0}$).
Theorem~\ref{thm:CoHo} implies that there is a unique automorphism of $A_{0}$ with these properties.
As the identity map from $A_0$ to $A_0$ also induces
the identity on $k_0$ and fixes $s_0(\bfbeta_{0})$, we conclude $\varphi_0 = \mathrm{id}_{A_0}$.
 
Finally, we show that if $\varphi_k$ is an isomorphism, then $\varphi$ is an isomorphism: if $\varphi_k$ is an isomorphism and $\bfbeta_{1}$ is a $p$-basis of $k_1$, then $\varphi_k(\bfbeta_{1})$ is 
a $p$-basis of $k_2$. Thus, $\varphi(A_1)$ contains the lift of a $p$-basis for $k_2$, and hence we have
$A_2 \subseteq \varphi(A_1)$ by Mac Lane's Identity Theorem
(Theorem~\ref{thm:MIT}).
\end{proof}

More explicitly, given an isomorphism between the residue fields of two Cohen rings of the same characteristic, we get a complete understanding of its lifts to isomorphisms of Cohen rings:

\begin{corollary}[{Cohen Structure Theorem, v.1}]\label{cor:Cohen_structure_1}
Let $(A_{1},k_{1})$ and $(A_{2},k_{2})$ be two Cohen rings of the same characteristic,
let $\varphi_{k}:k_{1}\longrightarrow k_{2}$ be an isomorphism of residue fields,
and
let $\bfbeta\subseteq k_{1}$ be a $p$-basis.
Consider representatives
$s_{1}:\bfbeta\longrightarrow A_{1}$
and
$s_{2}:\varphi_{k}(\bfbeta)\longrightarrow A_{2}$.
There exists a unique isomorphism of Cohen rings
$$
    \varphi=(\varphi_{A},\varphi_{k}):(A_{1},k_{1})\longrightarrow(A_{2},k_{2}),
$$
which respects $s_{1}$ and $s_{2}$,
and which is $\varphi_{k}$ on the residue fields.
\end{corollary}
\begin{proof}
If both $(A_{1},k_1)$ and $(A_{2},k_2)$ are strict then both existence and uniqueness follow from
Theorem \ref{thm:CoHo}.
Suppose next that both $(A_{1},k_1)$ and $(A_{2},k_2)$ are of characteristic $p^{m}$.
Let $(B,k_1)$ be a strict Cohen ring with representatives $s:\bfbeta\longrightarrow B$.
By Theorem \ref{thm:CoHo} there are unique morphisms
$\varphi^{1}=(\varphi^{1}_{A},\mathrm{id}_{k_1}):(B,k_1)\longrightarrow(A_{1},k_1)$
and $\varphi^{2}=(\varphi^{2}_{A},\varphi_{k}):(B,k_1)\longrightarrow(A_{2},k_2)$
which respect $s$ and $s_1$ (resp.~$s_{2}$) and which induce the identity (resp.~$\varphi_k)$ on the residue field.
Moreover, both $\varphi^{i}_{A}$ are surjective and both factor through the quotient map $B\longrightarrow B/\mathfrak{m}^{m}$ (cf.~Remark~\ref{rem:quotient}).
Thus, by the Isomorphism Theorem, both $(A_{i},k_i)$ are isomorphic to $(B/\mathfrak{m}^{m},k_1)$.
Therefore there is an isomorphism between them that respects $s_{1}$ and $s_{2}$, and induces $\varphi_{k}$ on the residue field.
This isomorphism is unique, by Theorem~\ref{thm:CoHo}.
\end{proof}

As long as one is only interested in the existence of an isomorphism of Cohen rings,
the following simplified version of the above is sufficient:

\begin{corollary}[{Cohen Structure Theorem, v.2}]\label{cor:Cohen_structure_2}
Let $(A_{1},k_1)$ and $(A_{2},k_2)$ be Cohen rings of the same characteristic, and 
assume that $\varphi_{k}:k_1 \longrightarrow k_2$ is an isomorphism of the residue fields.
There exists an isomorphism of Cohen rings
$$
    \varphi=(\varphi_{A},\varphi_k):(A_{1},k_1)\longrightarrow(A_{2},k_2),
$$
which is $\varphi_k$ on the residue fields.
\end{corollary}
\begin{proof}
Immediate from Corollary \ref{cor:Cohen_structure_1}.
\end{proof}

Our aim is now to apply Cohen's Homomorphism Theorem to give a clear statement of the relative structure of Cohen rings.
That is, we will describe the morphisms between Cohen rings 
which extend a given morphism between subrings. Although we will not refer to the statement later on in this paper, we state and prove it for future reference.
It should be noted once again that this is closely based on the work of Teichm\"{u}ller, Mac Lane, Cohen, and others.
See for example \cite{Tei36b}, \cite{Mac39c}, and \cite{Coh46}.

\begin{theorem}[Relative Homomorphism Theorem]
\label{thm:relative_homomorphism_theorem}
Let $(A_{1},k_{1})\subseteq(A_{2},k_{2})$
and
$(B_{1},l_{1})\subseteq(B_{2},l_{2})$
be two extensions of Cohen rings,
let
\begin{align*}
    \varphi=(\varphi_{A},\varphi_{k}):(A_{1},k_{1})\longrightarrow(B_{1},l_{1})
\end{align*}
be a morphism,
and let $\rmPhi_{k}:k_{2}\longrightarrow l_{2}$ be an embedding of fields which extends $\varphi_{k}$.
Suppose that both
$l_{2}/\rmPhi_{k}(k_{2})$
and
$k_{2}/k_{1}$
are separable.
Let
$\bfbeta$
be a $p$-basis of $k_{2}$ over $k_{1}$, 
and let $s_{A}:\bfbeta\longrightarrow A_{2}$
and $s_{B}:\rmPhi_{k}(\bfbeta)\longrightarrow B_{2}$
be choices of representatives.

Then, there exists a 
unique morphism of Cohen rings
$$
    \rmPhi:=(\rmPhi_{A},\rmPhi_{k}):(A_{2},k_{2})\longrightarrow(B_{2},l_{2}),
$$
that respects $s_{A}$ and $s_{B}$,
that induces $\rmPhi_{k}$ on the residue fields,
and that extends $\varphi$.
\end{theorem}

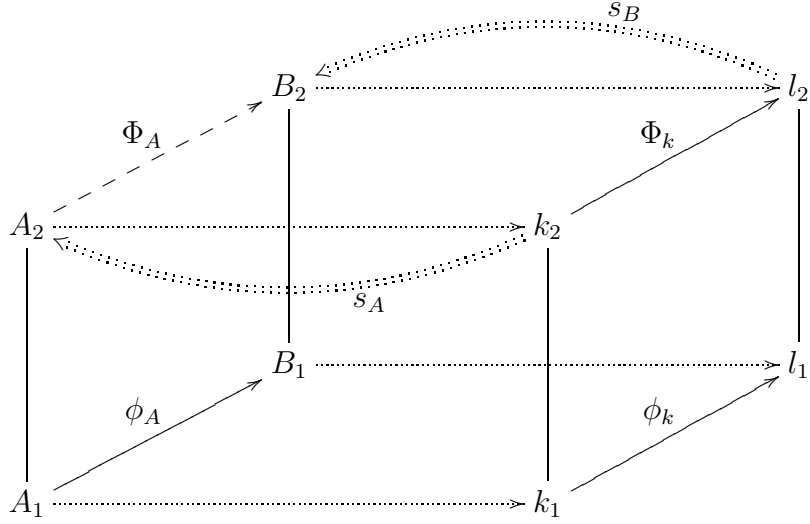
\begin{figure}[ht]
$$
\renewcommand{\labelstyle}{\textstyle}
    \xymatrix@=3em{
        &&
        B_{2}\ar@{.>}[rrrr]
        &&&&
        l_{2}
        \ar@{:>}@/_2pc/[llll]_(0.35){s_{B}}
        \\
        A_{2}\ar@{.>}[rrrr]
        \ar@{-->}[rru]_{}^{\rmPhi_{A}}
        &&&&
        k_{2}\ar@{->}[rru]_{}^{\rmPhi_{k}}
        \ar@{:>}@/^2pc/[llll]^(0.35){s_{A}}
        &&
        \\
        &&
        B_{1}\ar@{-}[uu]\ar@{.>}[rrrr]
        &&&&
        l_{1}\ar@{-}[uu]
        \\
        A_{1}\ar@{-}[uu]\ar@{->}[rru]^{\varphi_{A}}\ar@{.>}[rrrr]
        &&
        &&
        k_{1}\ar@{-}[uu]\ar@{->}[rru]^{\varphi_{k}}
        &&
    }
$$
\caption{Illustration of Theorem~\ref{thm:relative_homomorphism_theorem}}
\label{fig:4}
\end{figure}

\begin{proof}
We are given a $p$-basis $\bfbeta$ of $k_{2}$ over $k_{1}$,
that is each $\beta\in\bfbeta$ is of degree $p$ over $k_{2}^{(p)}k_{1}(\bfbeta\setminus\{\beta\})$,
and $k_{2}=k_{2}^{(p)}k_{1}(\bfbeta)$.
Choose any $p$-basis $\bfbeta_{A,1}$ of $k_{1}$ and any representatives
$s_{A,1}:\bfbeta_{A,1}\longrightarrow A_{1}$.
Since $k_{2}/k_{1}$ is separable,
$\bfbeta_{A,2}:=\bfbeta\sqcup\bfbeta_{A,1}$
is a $p$-basis of $k_{2}$.
We define
\begin{align*}
    s_{A,2}:\bfbeta_{A,2}&\longrightarrow A_{2}\\
    \beta&\longmapsto\left\{
    \begin{array}{ll}
    s_{A,1}(\beta)&\beta\in\bfbeta_{A,1}\\
    s_{A}(\beta)&\beta\in\bfbeta,
    \end{array}\right.
\end{align*}
which is a choice of representatives for $\bfbeta_{A,2}$.
Next we let $\bfbeta_{B,2}:=\varphi_{k}(\bfbeta_{A,1})\sqcup\rmPhi_{k}(\bfbeta)$.
We define
\begin{align*}
    s_{B,2}:\bfbeta_{B,2}&\longrightarrow B_{2}\\
    \beta&\longmapsto\left\{
    \begin{array}{ll}
    \varphi_{A}(s_{A,1}(\varphi_{k}^{-1}(\beta)))&\beta\in\varphi_{k}(\bfbeta_{A,1})\\
    s_{B}(\beta)&\beta\in\rmPhi_{k}(\bfbeta),
    \end{array}\right.
\end{align*}
which is a choice of representatives for $\bfbeta_{B,2}$.
It follows from Theorem \ref{thm:CoHo} that there is a unique morphism
$$
    \rmPhi=(\rmPhi_{A},\rmPhi_{k}):(A_{2},k_{2})\longrightarrow(B_{2},k_{2}),
$$
which respects $s_{A,2}$ and $s_{B,2}$, and which is $\rmPhi_{k}$ on the residue fields.
Observe that $\rmPhi$ extends $\varphi$ since in particular $\rmPhi$ respects $s_{A,1}$ and $\varphi_{A}\circ s_{A,1}\circ\varphi_{k}^{-1}|_{\varphi_{k}(\bfbeta_{A,1})}$ (the latter being a choice of representatives for $\varphi_{k}(\bfbeta_{A,1})$).
This proves the existence part of our claim.
For uniqueness, if $\rmPsi$ is any other morphism which extends $\varphi$ and respects $s_{A}$ and $s_{B}$ then we may argue that it also respects $s_{A,2}$ and $s_{B,2}$, just as for $\rmPhi$.
Therefore $\rmPhi=\rmPsi$ by Theorem \ref{thm:CoHo}.
\end{proof}

\section{Cohen--Witt rings}

Let $k$ denote a field of characteristic $p>0$.
For each natural number $n\in\mathbb{N}$, we denote the {\bf $n$-th Witt ring} over $k$ by 
$W_{n+1}(k)$,
and the {\bf infinite Witt ring} 
we denote by $W[k]$,
as described, for example, in \cite{vdD14} and in many other places.

If $k$ is perfect, then $W[k]$ is a complete discrete valuation ring of characteristic zero with residue field $k$.
That is, $(W[k],k)$ is a strict Cohen ring.
By Theorem \ref{thm:CoHo}, $(W[k],k)$ may be viewed as providing the canonical example of a Cohen ring with residue field $k$, canonical in the sense that for perfect $k$ there is a canonical isomorphism between any two strict Cohen rings with residue field $k$.
Likewise, $(W_{n}(k),k)$ is the canonical example of a Cohen ring with residue field $k$, of characteristic $p^{n}$.

On the other hand,
if $k$ is imperfect, then $W[k]$ fails to be a valuation ring.
There is a less well-known construction,
appropriate for the case of imperfect residue fields, which constructs Cohen rings as subrings of Witt rings
(see e.g.~\cite{Sch72}).
To mitigate the conflict with our own terminology,
we will refer to these more concrete rings as
`Cohen--Witt rings'.
We fix a $p$-basis $\bfbeta$ of $k$.
For each $n\in\mathbb{N}$,
the {\bf $n$-th Cohen--Witt ring} over $k$,
which we denote by $C_{n+1}(k)$,
is the subring of $W_{n+1}(k)$ generated by $W_{n+1}(k^{(p^{n})})$ and the elements $[\beta]=(\beta,0,...)$, for $\beta\in\bfbeta$.
That is
\begin{align*}
C_{n+1}(k)&:=W_{n+1}(k^{(p^{n})})\big([\beta]\mid\beta\in\bfbeta\big).
\end{align*}
We note that $C_{n+1}(k)$ is a local ring, with maximal ideal $(p)$ and residue field $k$.
Thus $(C_{n+1}(k),k)$ is indeed a Cohen ring.
There are representatives
$s_{n}:\bfbeta\longrightarrow C_{n+1}(k)$,
given by $s_{n}(\beta)=[\beta]$, for $\beta\in\bfbeta$.
The maps $\pi_{n}:W_{n+1}(k)\longrightarrow W_{n}(k)$, which are given by the truncation of the Witt vectors, restrict to surjections
\begin{align*}
\pi_{n}|_{C_{n+1}(k)}:C_{n+1}(k)&\longrightarrow C_{n}(k).
\end{align*}
Just as with the Witt rings, the Cohen--Witt rings equipped with these maps form an inverse system, as in Remark~\ref{rem:inverse_system}, the inverse limit of which is the {\bf strict Cohen--Witt ring} over $k$:
\begin{align*}
C[k]&:=\lim_{\longleftarrow}C_{n+1}(k).
\end{align*}
The field of fractions of $C[k]$ is often denoted by $C(k)$.
This system may be enriched with a compatible system of representatives
$s_{n}:\bfbeta\longrightarrow C_{n+1}(k)$,
as in Remark~\ref{rem:inverse_system_2}.
It is a consequence of
Corollary \ref{cor:Cohen_structure_2}
that any strict Cohen ring $(A,k)$ is isomorphic to the strict Cohen--Witt ring $C[k]$, though the isomorphism is not canonical in the sense that it depends on our choices of $\bfbeta$ and $s$.

\part{The Model Theory}

\section{Theories and Completeness}

Having developed the algebraic theory of Cohen rings, \label{part:2}
we are now in a position to describe their first-order theories.
Let $\mathfrak{L}_{\mathrm{ring}}=\{+,-,\cdot,0,1\}$ denote the first-order language of rings,
and
$\mathfrak{L}_{\mathrm{vf}} = \mathfrak{L}_\mathrm{ring}\cup \{\mathcal{O}\}$
the expansion of $\mathfrak{L}_\mathrm{ring}$ by a unary predicate (usually interpreted as the valuation ring).
To study the structure of the value group, we also consider the language of ordered abelian groups
$\mathfrak{L}_{\mathrm{oag}}=\{0,+,\leq\}$.

We consider the following theories:

\begin{definition}
Let $k$ be a field of characteristic $p$.
\begin{itemize}
\item 
Let $T_\mathrm{pc}$ be the $\mathfrak{L}_{\mathrm{ring}}$-theory of commutative rings with unity in which $(p)$ is the unique maximal ideal.
\item
Let $T_\mathrm{pc}(k,n)$ be the $\mathfrak{L}_{\mathrm{ring}}$-theory consisting of the union of $T_\mathrm{pc}$ with axioms that assert of a model $B$ that its characteristic is $p^{n}$ and that its residue field $k_{B}$ is a elementarily equivalent to $k$.
\item Let $T_\mathrm{ur}$ be the $\mathfrak{L}_{\mathrm{vf}}$-theory that asserts of a model $(K,v)$ that $v$ is henselian, $\mathcal{O}_{v}$ is a valuation ring on $K$ which is a model of $T_\mathrm{pc}$, and that the characteristic of $K$ is zero.
\item 
Let $T_\mathrm{ur}(k,\rmGamma)$ be the $\mathfrak{L}_{\mathrm{vf}}$-theory extending $T_\mathrm{ur}$ which requires in addition for a model $(K,v)$ that $Kv$ is elementarily equivalent to $k$ and that $vK$ is elementarily equivalent to $\rmGamma$.
\end{itemize}
\end{definition}

Note that
if $\rmGamma$ is an ordered abelian group \emph{without minimum positive element}
then $T_\mathrm{ur}(k,\rmGamma)$ is not consistent.
We do not define $T_{\mathrm{ur}}(k,\rmGamma)$ for a field $k$ of characteristic $0$.

The aim of this section is to show that $T_\mathrm{pc}(k,n)$ and $T_\mathrm{ur}(k,\rmGamma)$ are complete and to deduce the usual AKE type consequences from this. The non-strict case is a simple application
of the Structure Theorem for Cohen rings of positive characteristic:
\begin{theorem}\label{thm:completeness_nonstrict}
For any field $k$ of positive characteristic, the $\mathfrak{L}_\mathrm{ring}$-theory
$T_\mathrm{pc}(k,n)$
is complete.
\end{theorem}
\begin{proof}
Let $B_{1},B_{2}\models T_\mathrm{pc}(k,n)$.
By the Keisler--Shelah Theorem (\cite{She71}), replacing $B_{1}$ and $B_{2}$ with suitable ultrapowers if necessary,
we may assume that
there is an $\mathfrak{L}_{\mathrm{ring}}$-isomorphism
$\varphi_{k}:k_{B_{1}}\longrightarrow k_{B_{2}}$.
Applying the Structure Theorem for Cohen rings
(Corollary \ref{cor:Cohen_structure_2}), we get an isomorphism 
$\varphi:B_{1}\longrightarrow B_{2}$
that induces $\varphi_{k}$.
In particular this implies that $B_{1}$ and $B_{2}$ are elementarily equivalent.
\end{proof}

For the general case, we use the usual proof method of combining the Structure Theorem with a coarsening argument which allows us to apply the Ax--Kochen/Ershov Theorem in the equicharacteristic $0$ case:
\begin{theorem}\label{thm:completeness}
For any field $k$ of positive characteristic and any ordered abelian group $\rmGamma$ with minimum positive element, the $\mathfrak{L}_\mathrm{vf}$-theory
$T_\mathrm{ur}(k,\rmGamma)$
is complete.
\end{theorem}
\begin{proof}
Let $(K_{1},v_{1}),(K_{2},v_{2})\models T_{\mathrm{ur}}(k,\rmGamma)$.
By the Keisler--Shelah Theorem (\cite{She71}), replacing each structure with a suitable ultrapower if necessary,
we may assume that
both $(K_{1},v_{1})$ and $(K_{2},v_{2})$ are $\aleph_{1}$-saturated,
and that there is an $\mathfrak{L}_\mathrm{ring}$-isomorphism
$\varphi_{k}:K_{1}v_{1}\longrightarrow K_{2}v_{2}$
and an $\mathfrak{L}_\mathrm{oag}$-isomorphism
$\varphi_{\rmGamma}:v_{1}K_{1}\longrightarrow v_{2}K_{2}$.
For $i=1,2$,
let $w_{i}$ denote the finest proper coarsening of $v_{i}$ (note that $w_i$ exists because
$v_iK_i$ has a minimum positive element)
and let $\bar{v}_{i}$
denote the valuation induced by $v_{i}$ on the residue field $K_{i}w_{i}$.

By $\aleph_{1}$-saturation,
both
valued fields $(K_{i}w_{i}, \bar{v_i})$ are spherically complete and hence the valuation rings 
$\mathcal{O}_{\bar{v}_{i}}$
are strict Cohen rings.
By the Structure Theorem for Cohen rings (Corollary \ref{cor:Cohen_structure_2}),
there exists an isomorphism
$\varphi:\mathcal{O}_{\bar{v}_{1}}\longrightarrow\mathcal{O}_{\bar{v}_{2}}$
that induces $\varphi_{k}$. 
Note that $\varphi_{\rmGamma}$ also induces an isomorphism
$\bar{\varphi}_{\rmGamma}:w_{1}K_{1}\longrightarrow w_{2}K_{2}$
because any isomorphism of ordered abelian groups sends a minimum positive element (and hence the generated convex subgroup) to another such.
Since both
$(K_{i},w_{i})$
are henselian of equicharacteristic zero,
the Ax--Kochen/Ershov principle (\cite[AKE-Theorem 5.1]{vdD14}) implies that
$(K_{1},w_{1})$ and $(K_{2},w_{2})$ are elementarily equivalent.

It follows by the $\emptyset$-definability of the valuation $v_{i}$ in $K_{i}$ (this is a variant of Robinson's classical definition of $\mathbb{Z}_p$ in $\mathbb{Q}_p$, see, e.g., 
\cite[Corollary 2]{Hon}) that
$(K_{1},v_{1})$ and $(K_{2},v_{2})$ are elementarily equivalent, as required.
\end{proof}

\begin{remark}
In fact, B\'elair proves relative quantifier elimination for the valued field sort of an unramified henselian valued field in the $\omega$-sorted language $\mathcal{L}_{\mathrm{co}_{\omega}}$ down to the sorts
for the residue rings $\mathcal{O}/\mathfrak{m}^n$ (cf.~\cite[Th\'eor\`{e}me 5.1]{Bel99}). Applying this quantifier
elimination,
Theorem \ref{thm:completeness} can also be deduced from Theorem \ref{thm:completeness_nonstrict}.
\end{remark}

In particular, Theorem \ref{thm:completeness} immediately implies the following Ax-Kochen/Ershov-type result, sometimes referred to as an AKE$_\equiv$-principle:
\begin{corollary}[Ax--Kochen/Ershov principle for unramified henselian valued fields] \label{cor:AKE}
Let $(K,v)$ and $(L,w)$ be two unramified henselian valued fields.
Then
$$
\underbrace{Kv \equiv Lw}_{\textrm{in }\mathfrak{L}_\mathrm{ring}}  \textrm{ and }  \underbrace{vK \equiv wL}_{\textrm{in }\mathfrak{L}_\mathrm{oag}} \,
\Longleftrightarrow \, 
\underbrace{(K,v)\equiv (L,w)}_{\textrm{in }\mathfrak{L}_\mathrm{vf}}.
$$
\end{corollary}

\begin{remark}
Corollary \ref{cor:AKE} above
is essentially claimed by B\'elair in 
\cite[Corollaire 5.2]{Bel99}.
However, B\'{e}lair's proof only goes through in the case of a perfect residue field, since it uses the rings of Witt vectors.
\end{remark}

The Ax--Kochen/Ershov Principle, above, immediately gives an axiomatisation of the complete theories of unramified henselian valued fields, as follows.

\begin{corollary}\label{cor:axiomatisation}
Let $(K,v)$ be an unramified henselian valued field of mixed characteristic.
The complete $\mathfrak{L}_{\mathrm{vf}}$-theory of $(K,v)$ is axiomatised by
\begin{enumerate}
\item $(K,v)$ is an henselian valued field of mixed characteristic $(0,p)$,
\item the value group is elementarily equivalent to $vK$ and $v(p)$ is minimum positive, and
\item the residue field is elementarily equivalent to $Kv$.
\end{enumerate}
\end{corollary}
In particular, we get an axiomatization of the $\mathfrak{L}_\mathrm{vf}$-theory of $(C(k),v)$, relative to the $\mathfrak{L}_\mathrm{ring}$-theory of the residue field $k$.

\section{Relative model completeness}

In analogy to the case of unramified henselian fields with perfect residue field (cf.~\cite[Theorem 7.2]{vdD14}), we also get
an AKE$_\preceq$-principle for the case of arbitrary residue fields.
To prove relative completeness, or in other words an AKE$_\equiv$-principle for unramified henselian
valued fields, it is sufficient to know that Cohen rings are unique up to isomorphism.
However, in order to prove the AKE$_\preceq$-principle for unramified henselian valued fields, we need to apply the
Relative Structure Theorem
(Corollary~\ref{cor:CoHo_relative}).
We first state the non-strict version:

\begin{proposition}[Relative model completeness, non-strict version]\label{prp:MC_nonstrict}
Given $A, B \models T_\mathrm{pc}(k,n)$ with $A \subseteq B$ such that the induced embedding of residue fields
$k_A \subseteq k_B$ is an elementary embedding (in $\mathfrak{L}_\mathrm{ring}$), we have
$$A \preceq B.$$
\end{proposition}
\begin{proof}
Let $B_{1},B_{2},A\models T_{\mathrm{pc}}(k,n)$ with $A\subseteq B_{1},B_{2}$ two extensions. 
We suppose that the induced extensions
$k_{A}\preceq k_{B_{i}}$, for $i=1,2$,
are elementary.
We claim that $B_1$ and $B_2$ are elementarily equivalent over $A$,
symbolically $B_{1}\equiv_{A}B_{2}$.

By the Keisler--Shelah Theorem (\cite{She71}),
replacing $B_{1}$ and $B_{2}$ with suitable ultrapowers if necessary,
we may assume that there is an $\mathfrak{L}_{\mathrm{ring}}$-isomorphism
$\varphi_{k}:k_{B_{1}}\longrightarrow k_{B_{2}}$
that fixes $k_{A}$ pointwise.
By Corollary~\ref{cor:CoHo_relative},
there is an isomorphism
$\varphi:(B_{1},k_{1})\longrightarrow(B_{2},k_{2})$
that induces $\varphi_{k}$ and fixes $A$ pointwise.
In particular, $B_{1}$ and $B_{2}$ are elementarily equivalent over $A$.
This proves the claim.

We return to the setting of an extension $A\subseteq B$ of models of $T_{\mathrm{pc}}(k,n)$ for which $k_{A}\preceq k_{B}$ is elementary.
From the claim it follows that $A\equiv_{A}B$, equivalently the extension $A\subseteq B$ is elementary.
\end{proof}

For the strict version of the relative model completeness theorem, we combine the Relative Structure Theorem (Corollary \ref{cor:CoHo_relative}) with the coarsening method and 
well-known results from the equicharacteristic zero world.

\begin{theorem}[Relative model completeness]\label{thm:MC_strict}
Given an extension $(K,v) \subseteq (L,w)$ of unramified henselian valued fields
such that the induced embeddings of residue fields $Kv \subseteq Lw$ and value groups $vK \subseteq wL$ are elementary (in $\mathfrak{L}_\mathrm{ring}$ and $\mathfrak{L}_\mathrm{oag}$ respectively), we have $$(K,v)\preceq (L,w).$$
\end{theorem}
\begin{proof}
Let $(K_{i},v_{i})\models T_{\mathrm{ur}}(k,\rmGamma)$, for $i=0,1,2$,
be such that
$(K_{0},v_{0})\subseteq(K_{1},v_{1}),(K_{2},v_{2})$
are two extensions of valued fields.
We suppose that the residue field extensions
$K_{0}v_{0}\preceq K_{i}v_{i}$,
and the value group extensions
$v_{0}K_{0}\preceq v_{i}K_{i}$,
both for $i=1,2$,
are elementary.
We claim that $(K_{1},v_{1})$ and $(K_{2},v_{2})$ are elementarily equivalent over $(K_{0},v_{0})$,
symbolically $(K_{1},v_{1})\equiv_{(K_{0},v_{0})}(K_{2},v_{2})$.

By the Keisler--Shelah Theorem (\cite{She71}),
replacing each valued field with a suitable ultrapower if necessary,
we may assume that all three
valued fields
are $\aleph_{1}$-saturated
and that there is an isomorphism
$\varphi_{k}:K_{1}v_{1}\longrightarrow K_{2}v_{2}$
that fixes $K_{0}v_{0}$ pointwise,
and an isomorphism
$\varphi_{\rmGamma}:v_{1}K_{1}\longrightarrow v_{2}K_{2}$
that fixes $v_{0}K_{0}$ pointwise.
For $i=0,1,2$,
let $\hat{v}_{i}$ be the finest proper coarsening of $v_{i}$,
and let $\bar{v}_{i}$ be the valuation induced on $K_{i}\hat{v}_{i}$ by $v_{i}$.
By $\aleph_{1}$-saturation,
all three
$(K_{i}\hat{v}_{i},\bar{v}_{i})$
are strict Cohen rings.
Note also that these coarsenings are compatible in the sense that,
for both $i=1,2$,
we have extensions
$(K_{0},\hat{v}_{0})\subseteq(K_{i},\hat{v}_{i})$
and
$(K_{0}\hat{v}_{0},\bar{v}_{0})\subseteq(K_{i}\hat{v}_{i},\bar{v}_{i})$.
Moreover $\varphi_{\rmGamma}$ induces an isomorphism
$\hat{\varphi}_{\rmGamma}:\hat{v}_{1}K_{1}\longrightarrow\hat{v}_{2}K_{2}$
that fixes $\hat{v}_{0}K_{0}$ pointwise,
since
$\varphi_{\rmGamma}$ restricts to an isomorphism between the convex subgroups
$\langle v_{1}(p)\rangle$
and 
$\langle v_{2}(p)\rangle$,
and since 
$\hat{v}_{i}K_{i}$ is the quotient of $v_{i}K_{i}$ by
$\langle v_{i}(p)\rangle$.

By Corollary~\ref{cor:CoHo_relative},
there is an isomorphism
$\varphi:(K_{1}\hat{v}_{1},\bar{v}_{1})\longrightarrow(K_{2}\hat{v}_{2},\bar{v}_{2})$
that induces $\varphi_{k}$ and fixes $K_{0}\hat{v}_{0}$ pointwise.
In particular $K_{1}\hat{v}_{1}$ and $K_{2}\hat{v}_{2}$ are elementarily equivalent over $K_{0}\hat{v}_{0}$.

Therefore $(K_{1},\hat{v}_{1})$ and $(K_{2},\hat{v}_{2})$ are two henselian valued fields of equicharacteristic zero, both extending $(K_{0},\hat{v}_{0})$,
with value groups isomorphic over $\hat{v}_{0}K_{0}$
and
residue fields elementarily equivalent over $K_{0}\hat{v}_{0}$.
It follows from
\cite[Theorem 7.1]{FVK}
that
$(K_{1},\hat{v}_{1})$ and $(K_{2},\hat{v}_{2})$
are elementarily equivalent over $(K_{0},\hat{v}_{0})$,
i.e.~$(K_{1},\hat{v}_{1})\equiv_{(K_{0},\hat{v}_{0})}(K_{2},\hat{v}_{2})$.
Since the valuation rings of all three valuations $v_{i}$
are $\emptyset$-definable by the same $\mathfrak{L}_{\mathrm{ring}}$-formula
in each field $K_{i}$ (again, this is a variant of Robinson's classical definition of $\mathbb{Z}_p$ in $\mathbb{Q}_p$, see, e.g., 
\cite[Corollary 2]{Hon}),
it follows that $(K_{1},v_{1})$ and $(K_{2},v_{2})$
are elementarily equivalent over $(K_{0},v_{0})$,
i.e.~$(K_{1},v_{1})\equiv_{(K_{0},v_{0})}(K_{2},v_{2})$.
This proves the claim.

We return to the setting of an extension $(K,v)\subseteq(L,w)$ of models of $T_{\mathrm{ur}}$ for which $Kv\preceq Lw$ and $vK\preceq wL$ are elementary.
From the claim it follows that $(K,v)\equiv_{(K,v)}(L,w)$, equivalently the extension $(K,v)\subseteq(L,w)$ is elementary.
\end{proof}

\begin{remark}
In fact, the relative model completeness for non-strict Cohen rings
(Proposition~\ref{prp:MC_nonstrict})
can also be combined
with a result by B\'elair (\cite[Corollaire 5.2(2)]{Bel99}) to show model completeness,
albeit in a slightly different ($\omega$-sorted) language.
\end{remark}

As a consequence, we get the following embedding version of the Ax-Kochen/Ershov result:
\begin{corollary}\label{cor:AKE-E}
Let $(K,v)\subseteq (L,w)$ be two unramified henselian valued fields.
Then, we have
$$ \underbrace{Kv \preceq Lw}_{\textrm{in }\mathfrak{L}_\mathrm{ring}} \textrm{ and } \underbrace{vK \preceq wL}_{\textrm{in }\mathfrak{L}_\mathrm{oag}} \, \Longleftrightarrow \, 
\underbrace{(K,v) \preceq (L,w)}_{\textrm{in }\mathfrak{L}_\mathrm{vf}}.
$$
\end{corollary}

For Cohen rings, or more generally unramified henselian valued fields taking values in a $\mathbb{Z}$-group, 
our result simplifies to:
\begin{corollary}
Let $(K,v)\subseteq (L,w)$ be two unramified henselian valued fields with value groups $vK\equiv wL\equiv\mathbb{Z}$.
Then, we have
$$ \underbrace{Kv \preceq Lw}_{\textrm{in }\mathfrak{L}_\mathrm{ring}}
\, \Longleftrightarrow \, 
\underbrace{(K,v) \preceq (L,w)}_{\textrm{in }\mathfrak{L}_\mathrm{vf}}.
$$
\end{corollary}

\section{Embedding lemma and relative existential completeness}
The aim of this section is to
prove an embedding lemma for unramified henselian valued fields. This will be applied to prove relative existential completeness of unramified
henselian valued fields of fixed finite degree of imperfection as well as to show stable embeddedness of residue field
and value group in the subsequent section.
The proof of the embedding lemma (i.e.~the next proposition) is a refined version
of the proofs given in \cite[Lemmas 5.6 and 6.4]{FVK}. In Kuhlmann's terminology,
we show that the class of $\aleph_1$-saturated models of $T_\mathrm{ur}$ satisfies an appropriate version of the 
Relative Embedding Property (cf.~\cite[p.~31]{FVK}).

\begin{proposition}[Embedding Lemma] \label{prp:emb}
Let $(L_{1},v_{1})$ and $(L_{2},v_{2})$ be extensions of $(K,v)$, and assume that all three
are $\aleph_1$-saturated models of $T_\textrm{ur}$.
Suppose that $L_{1}v_{1}/Kv$ is separable and $v_{1}L_{1}/vK$ is torsion-free.
Moreover, assume that $(L_{2},v_{2})$ is $|L_{1}|^{+}$-saturated and that there
are embeddings
$\varphi_{k}:L_{1}v_{1}\longrightarrow L_{2}v_{2}$
over $Kv$ 
and
$\varphi_{\rmGamma}:v_{1}L_{1}\longrightarrow v_{2}L_{2}$
over $vK$.
Suppose that $L_{2}v_{2}/\varphi_{k}(L_{1}v_{1})$ is separable.
Then, there is an embedding $\varphi:(L_1,v_1) \longrightarrow (L_2,v_2)$ over $K$
that induces $\varphi_k$ and $\varphi_\rmGamma$ on residue field and value group respectively.

Moreover, if $\varphi_{k}$ and $\varphi_{\rmGamma}$ are elementary embeddings, then any such embedding $\varphi$ is elementary.
\end{proposition}
\begin{proof}
Since $(K,v)$ is henselian and unramified, it is defectless, therefore our assumptions on value group and residue field imply that $K$ is relatively algebraically closed in $L_{1}$.
Let $w$ denote the finest proper coarsening of $v$ on $K$, and, likewise, let $w_i$ denote 
the finest proper coarsening of $v_i$ on $L_i$ (for $i=1,2$).
Note that $w$ is the restriction of each $w_{i}$ to $K$.
By our saturation assumption,
the valued fields $(Kw, \bar{v})$ and $(L_iw_i, \bar{v_i})$ are all Cohen fields.
The inclusion $(K,w)\subseteq(L_{1},w_1)$ gives rise to an inclusion $Kw \subseteq L_1w_1$. Moreover, $\varphi_k$ induces an (not necessarily unique) embedding $$\psi_k: (L_1w_1, \bar{v_1}) \longrightarrow
(L_2w_2, \bar{v_2})$$ over $(Kw,\bar{v})$ by Cohen's Homomorphism theorem (Theorem \ref{thm:CoHo}). 
 We now fix one such $\psi_k$.
We note that, in order to show that an embedding $\varphi$ of a subfield
of $(L_1,v_1)$ into $(L_2,v_2)$ induces $\varphi_k$, it suffices to show that it
induces
$\psi_k$:
\begin{claim} \label{psi}
For any $(F,w_1,v_{1})\subseteq(L_{1}, w_1 ,v_{1})$ extending $(K,w,v)$,
if $\varphi:(F,w_{1})\longrightarrow(L_{2},w_{2})$ is an embedding over $K$ that induces $\psi_{k}$,
then $\varphi$ is also an embedding
$\varphi:(F,v_{1})\longrightarrow(L_{2},v_{2})$
that induces $\varphi_{k}$.
\end{claim}
\begin{claimproof}
Note that $\psi_{k}$ is an embedding of valued fields $(L_{1}w_{1},\bar{v}_{1})\longrightarrow(L_{2}w_{2},\bar{v}_{2})$.
For $a\in F$, we have
\begin{align*}
    v_{1}(a)\geq0&\Leftrightarrow w_{1}(a)\geq0\;\textrm{and}\;\bar{v}_{1}(aw_{1})\geq0\\
    &\Leftrightarrow w_{2}(\varphi(a))\geq0\;\textrm{and}\;\bar{v}_{2}(\psi_{k}(aw_{1}))\geq0\\
    &\Leftrightarrow w_{2}(\varphi(a))\geq0\;\textrm{and}\;\bar{v}_{2}((\varphi(a))w_{2})\geq0\\
    &\Leftrightarrow v_{2}(\varphi(a))\geq0.
\end{align*}
Since $\varphi$ induces $\psi_{k}$, and $\psi_{k}$ is an embedding
$(Fw_{1},\bar{v}_{1})\longrightarrow(L_{2}w_{2},\bar{v}_{2})$
that induces $\varphi_{k}$,
it follows that $\varphi$ induces $\varphi_{k}$.
\end{claimproof}

We now adapt the proof of \cite[Lemmas 5.6 and 6.4]{FVK} carefully to our setting, in order to construct an
embedding $\varphi: (L_1,v_1) \longrightarrow (L_2,v_2)$
over $K$
that induces
$\varphi_k$ and $\varphi_\rmGamma$.
We also use $\varphi$ to denote the restriction of $\varphi$ to any subfield of $L_1$.
Let $\mathcal{T}$ be a standard valuation transcendence basis of $(L_{1},w_{1})/(K,w)$, i.e.
$$\mathcal{T} =\{x_i, y_i\mid i\in I, j\in J \}$$
such that the set of values $\{w_1(x_i)\}_{i\in I}$ is a maximal rationally independent set in 
$w_1L$ over $wK$ and such that the set of residues $\{y_jw_1\}_{j \in J}$ is a transcendence base
of $L_1w_1$ over $Kw$.
Let $K'$ be the relative algebraic closure of $K(\mathcal{T})$ in $L_1$, and, by an abuse of notation,
let $w_1$ also denote the restriction of $w_1$ to $K'$ and its subfields.
Note that $(K',w_{1})/(K,w)$ is without transcendence defect, by \cite[Corollary 2.4]{FVK}.
See \cite[p.~4]{FVK} for the definition of `without transcendence defect'.

\smallskip
{\bf Step 1: Extending to valued function fields without transcendence defect.}
Let $L$ be a subfield of $K'$ that is a finitely generated extension of $K$.
Since $K$ is relatively algebraically closed in $L$, $(L,w_{1})/(K,w)$ is a valued function field without transcendence defect.
By \cite[Theorem 1.9]{FVK}, $(L,w_{1})/(K,w)$ is strongly inertially generated,
i.e.~there is a transcendence basis $\mathcal{T}_L=\{x_{i},y_{j}\mid i\in I_L,j\in J_L\}\subseteq L$ such that
\begin{enumerate}
\item
$w_{1}L=w_{1}K(\mathcal{T}_L)=wK\oplus\bigoplus_{i}\mathbb{Z}\cdot w_{1}(x_{i})$
\item
$\{y_{j}w_{1}\}_{j \in J_L}$ is a separating transcendence base of $Lw_{1}/Kw$, and
\item
there is an element $a\in L^{h}$ (the henselization of $L$ with respect to $w_{1}$) such that
$L^{h}=K(\mathcal{T}_{L})^{h}(a)$,
$w_{1}(a)=0$,
and
$\big(K(\mathcal{T}_{L})w_{1}\big)(aw_{1})/K(\mathcal{T}_{L})w_{1}$
is a separable extension of degree 
$[K(\mathcal{T}_{L})^{h}(a):K(\mathcal{T}_{L})^{h}]$.
\end{enumerate}
Note that in our case, the separability in {\bf(ii)} and {\bf(iii)} is automatic, since $w_{1}$ is of residue characteristic zero.
We now explore these three properties, one after the other, in order to construct an embedding $\varphi:(L,v_{1})\longrightarrow(L_{2},v_{2})$ over $K$ that induces $\varphi_{k}$ and $\varphi_{\rmGamma}$.
\begin{claim}
$v_{1}L=vK\oplus\bigoplus_{i}\mathbb{Z}\cdot v_{1}(x_{i})$.
\end{claim}
\begin{claimproof}
Suppose there are $n_{i}\in\mathbb{Q}$ and $\alpha\in vK$ such that 
$\alpha+\sum_{i}n_{i}\cdot v_{1}(x_{i})=0$.
Then
$(\alpha+\mathbb{Z})+\sum_{i}n_{i}\cdot w_{1}(x_{i})=0+\mathbb{Z}$.
By $\mathbb{Q}$-linear independence of the $w_{1}(x_{i})$ over $wK$ in $w_{1}L$,
all the $n_{i}$ are zero.
Therefore the $v_{1}(x_{i})$ are $\mathbb{Q}$-linearly independent from $vK$ in $v_{1}L$.
Thus $vK\oplus\bigoplus_{i}\mathbb{Z}\cdot v_{1}(x_{i})\leq v_{1}L$.

Let $\gamma\in v_{1}L$.
There exist $n_{i}\in\mathbb{Z}$ and $\alpha\in vK$ such that
$(\gamma+\mathbb{Z})=(\alpha+\mathbb{Z})+\sum_{i}n_{i}w_{1}(x_{i})$.
Therefore for some $\beta\in vK$ we have
$\gamma=\beta+\sum_{i}n_{i}v_{1}(x_{i})$.
\end{claimproof}
Next, we choose $\mathcal{T}_{L}'=\{x_{i}',y_{j}'\mid i\in I_{L},j\in J_{L}\}\subseteq L_{2}$ such that
\begin{enumerate}
\item
$v_{2}(x_{i}')=\varphi_{\rmGamma}(v_{1}(x_{i}))$, for each $i\in I_{L}$, and
\item
$y_{j}'w_{2}=\psi_{k}(y_{j}w_{1})$, for each $j\in J_{L}$.
\end{enumerate}
Immediately:
$w_{2}(x_{i}')=\psi_{\rmGamma}(w_{1}(x_{i}))$, for each $i\in I_L$.
Exactly as in the proof of \cite[Lemma 5.6]{FVK},
there is an isomorphism
$\varphi:(K(\mathcal{T}_{L}),w_{1})\longrightarrow(K(\mathcal{T}_{L}'),w_{2})$
that maps
\begin{align*}
    x_{i}&\longmapsto x_{i}'\\
    y_{j}&\longmapsto y_{j}',
\end{align*}
and which therefore induces
$\psi_{k}$
and
$\psi_{\rmGamma}$.

\begin{claim}
$\varphi$ is also an isomorphism $\varphi:(K(\mathcal{T}_{L}),v_{1})\longrightarrow(K(\mathcal{T}_{L}'),v_{2})\subseteq(L_{2},v_{2})$
which induces
$\varphi_{k}$
and
$\varphi_{\rmGamma}$.
\end{claim}
Let $f\in K[\mathcal{T}_{L}]$, written
$f=\sum_{k}c_{k}\prod_{i}x_{i}^{\mu_{k,i}}\prod_{j}y_{j}^{\nu_{k,j}}$,
for $c_{k}\in K$.
\begin{claimproof}
We already know that $\varphi$ is an isomorphism of fields, inducing both $\psi_K$ and 
$\psi_\rmGamma$.
By Claim \ref{psi}, $\varphi$ is moreover an isomorphism of valued fields with respect to the $v_i$'s that induces $\varphi_{k}$.

We now check that it induces $\varphi_{\rmGamma}$,
first working in the case $f\in K[x_{i}\mid i\in I_L]$.
\begin{align*}
    v_{2}(\varphi(f))&=\min_{k}\{v_{2}(c_{k})+\sum_{i}\mu_{k,i}v_{2}x_{i}'\}\\
    &=\min_{k}\{v_{2}(c_{k})+\sum_{i}\mu_{k,i}\varphi_{\rmGamma}(v_{1}x_{i})\}\\
    &=\varphi_{\rmGamma}(\min_{k}\{v_{1}(c_{k})+\sum_{i}\mu_{k,i}v_{1}x_{i}\})\\
    &=\varphi_{\rmGamma}(v_{1}(f)).
\end{align*}
Since the value group of $K[x_{i}\mid i\in I_L]$ with respect to $v_{1}$ is already $vK\oplus\bigoplus_{i}\mathbb{Z}\cdot v_{1}x_{i}=v_{1}K(\mathcal{T}_{L})$, indeed $\varphi$ induces $\varphi_{\rmGamma}$.
\end{claimproof}
By the universal property of henselizations,
$\varphi$ extends to an isomorphism
$$\varphi: (K(\mathcal{T}_{L})^{h},w_{1})\longrightarrow(K(\mathcal{T}_{L}')^{h},w_{2}) \subseteq (L_2,w_2)$$
where the henselizations are taken with respect to $w_i$ (for $i \in \{1,2\}$). 
Note that $\varphi$ still induces both $\psi_{k}$ and $\psi_{\rmGamma}$ since henselizations are immediate extensions. Thus, by Claim \ref{psi}, $\varphi$ is also an
isomorphism $(K(\mathcal{T}_{L})^{h},v_{1})\longrightarrow
(K(\mathcal{T}_{L}')^{h},v_{2})$.
Since the henselization $K(\mathcal{T}_{L})^{h}$ of $K(\mathcal{T}_{L})$ with respect to $w_1$
is a subfield of the henselization with respect to $v_1$,
which is an immediate extension of $(K(\mathcal{T}_{L}),v_{1})$,
$\varphi$ induces $\varphi_{k}$ and $\varphi_{\rmGamma}$.

Recall that by property {\bf(iii)} of strong inertial generation, there exists an element $a\in L^{h}$ such that
$L^{h}=K(\mathcal{T}_{L})^{h}(a)$, $w_{1}(a)=0$,
and $(K(\mathcal{T}_{L})^{h}w_{1})(aw_{1})/K(\mathcal{T}_{L})^{h}w_{1}$ is a separable extension of degree
$[K(\mathcal{T}_{L})^{h}(a):K(\mathcal{T}_{L})^{h}]$.

Let $f\in\mathcal{O}_{(K(\mathcal{T}_{L})^{h},w_{1})}[X]$ be such that $fw_{1}$ is the minimal polynomial of $aw_{1}$ over $K(\mathcal{T}_{L})^{h}w_{1}$.
By henselianity, there exists a unique root $a'\in L_{2}$ of $\varphi(f)$ such that $a'w_{2}=\psi_{k}(aw_{1})$.
\begin{claim}
$\varphi$ extends to an isomorphism
$\varphi:(K(\mathcal{T}_{L})^{h}(a),v_{1})\longrightarrow
(K(\mathcal{T}_{L}')^{h}(a'),v_{2})\subseteq
(L_{2},v_{2})$
which maps $a$ to $a'$ and induces $\varphi_{k}$ and $\varphi_{\rmGamma}$.
\end{claim}
\begin{claimproof}
Mapping $a\longmapsto a'$ allows us to extend $\varphi$
to an isomorphism of fields
$\varphi:K(\mathcal{T}_{L})^{h}(a)\longrightarrow K(\mathcal{T}_{L}')^{h}(a')$.
By henselianity of $(K(\mathcal{T}_{L})^{h},w_{1})$,
$\varphi$ is an isomorphism of valued fields
$(K(\mathcal{T}_{L})^{h}(a),w_{1})\longrightarrow(K(\mathcal{T}_{L}')^{h}(a'),w_{2})$.

We now argue that $\varphi$ induces $\psi_{k}$.
Since
$1,aw_{1},\ldots,a^{n-1}w_{1}$
are linearly independent over $K(\mathcal{T}_{L})^{h}w_{1}$,
we have
$w_{1}\big(\sum_{i<n}c_{i}a^{i}\big)=\min_{i} w_{1}(c_{i})$,
for all $c_{i}\in K(\mathcal{T}_{L})^{h}$.
In particular, we have
$$\mathcal{O}_{(K(\mathcal{T}_{L})^{h}(a),w_{1})}
=\Big\{\sum_{i<n}c_{i}a^{i}\;\Big|\; c_{i}\in\mathcal{O}_{(K(\mathcal{T}_{L})^{h},w_{1})}\Big\}.$$
Let $g\in\mathcal{O}_{(K(\mathcal{T}_{L})^{h},w_{1})}[X]$.
Then $\varphi$ induces $\psi_{k}$ since we have
$$(\varphi(g(a)))w_{2}=(\varphi(g)(a'))w_{2}=(\varphi(g)w_{2})(a'w_{2})=(\psi_{k}(gw_{1}))(\psi_k(aw_{1}))=\psi_k((g(a))w_1).$$
Thus, by Claim \ref{psi},
$\varphi: (K(\mathcal{T}_{L})^{h}(a),v_{1})\longrightarrow
(K(\mathcal{T}_{L}')^{h}(a'),v_{2})$ is an isomorphism that induces $\varphi_k$.
Finally, the value group of $(K(\mathcal{T}_{L})^{h}(a),v_{1})$ is torsion over the value group of $(K(\mathcal{T}_{L})^{h},v_{1})$: for each $\gamma\in v_{1}K(\mathcal{T}_{L})^{h}(a)$ we have 
$n\gamma\in v_{1}K(\mathcal{T}_{L})^{h}$.
Since the restriction of $\varphi$ to $K(\mathcal{T}_{L})^{h}$ induces $\varphi_{\rmGamma}$ on $v_{1}K(\mathcal{T}_{L})^{h}$, and multiplication by $n$ in ordered abelian groups is injective,
we conclude that $\varphi$ induces $\varphi_{\rmGamma}$.
\end{claimproof}
This finishes the construction of an embedding
$\varphi:(L,v_{1})\longrightarrow(L_{2},v_{2})$
that induces $\varphi_{k}$ and $\varphi_{\rmGamma}$. Therefore, we have completed Step 1.
\smallskip

By saturation, we obtain an embedding
$$\varphi:(K',v_{1})\longrightarrow(L_{2},v_{2})$$ over $K$
that induces $\varphi_{k}$ and $\varphi_{\rmGamma}$. More precisely, we realise the
finitely consistent type
$$\mathrm{tp}_{\mathrm{qf}}((K',v_1)/(K,v))\cup \{v(x_{c})=\varphi_{\rmGamma}(v_{1}(c))\mid c\in K'\}\cup\{x_{c}v=\varphi_{k}(cv_{1})\mid c\in \mathcal{O}_{(K',v_{1})}\},$$
which is a type over $|K'|$-many parameters.

\smallskip
{\bf Step 2: Extending to immediate function fields.}
Our present aim is to extend $\varphi$ to an embedding $(L_{1},v_{1})\longrightarrow(L_{2},v_{2})$
over $K$
that 
induces $\varphi_{k}$ and $\varphi_{\rmGamma}$.
Since $K'$ contains a standard valuation transcendence basis for $(L_{1},w_{1})/(K,w)$,
we have that $w_{1}L_{1}/w_{1}K'$ is torsion and $L_{1}w_{1}/K'w_{1}$ is algebraic.
Since $K'$ is relatively algebraically closed in $L_{1}$, it follows that
$(K',v_{1})$ and $(K',w_1)$ are henselian.
Because $\mathrm{char}(K'w_1)=0$, we conclude that the extension $(L_1,w_1)/(K',w_1)$ is immediate
(for example see \cite[Lemma 3.7]{FVK}).
Thus, so is $(L_1,v_1)/(K',v_1)$.
Therefore any extension of $\varphi$ to an embedding
$(L_1,v_1) \longrightarrow (L_2,v_2)$ automatically induces $\varphi_k$ and $\varphi_\rmGamma$.

Consider a finitely generated subextension $F/K'$ of $L_{1}/K'$;
then $(F,v_{1})/(K',v_{1})$ is an immediate function field.
In fact, proceeding iteratively, it suffices to find a way to extend $\varphi$ to immediate function fields of transcendence degree $1$.
By \cite[Theorem 2.2]{Kuh19},
such extensions are henselian rational,
i.e.~subextensions of the henselization of a simple transcendental and immediate extension, of transcendental type.

Suppose we have extended $\varphi$ to an embedding
$\varphi:(F_{0},v_{1})\longrightarrow(L_{2},v_{1})$
over $K$ that induces $\varphi_{k}$ and $\varphi_{\rmGamma}$,
where $(F_{0},v_{1})$ is the henselization of a finitely generated extension of $K'$.
Let $b\in L_{1}$ be transcendental over $F_{0}$.
Then $b$ is a pseudo-limit of a pseudo-Cauchy sequence $(b_{\rho})_{\rho<\sigma}\subseteq F_{0}$ of transcendental type, with respect to $v_{1}$: as $(F_{0},v_1)$ is henselian and algebraically maximal, all its immediate extensions are of transcendental type.
Then $(\varphi(b_{\rho}))_{\rho<\sigma}$ is a pseudo-Cauchy sequence in $L_{2}$, also of transcendental type, now with respect to $v_{2}$.
By saturation, this sequence has a pseudo-limit $b'\in L_{2}$.
We extend $\varphi$ by mapping $b\longmapsto b'$
to an isomorphism of fields
$F_{0}(b)\longrightarrow\varphi(F_{0})(b')$.
Automatically, we get that 
$$\varphi:(F_{0}(b),v_{1})\longrightarrow(\varphi(F_{0})(b'),v_{2})\subseteq (L_2,v_2)$$
is an isomorphism of valued fields
since pseudo-Cauchy sequences of transcendental type determine the isomorphism type of the valued field generated by a pseudo-limit
(see \cite[Theorem 2]{Kaplansky}).

Finally, we extend $\varphi$ to the henselization of $(F_0(b),v_{1})$, which
is a subfield of $(L_1,v_1)$. This is accomplished by applying the universal property of the henselization once again.

\smallskip
We have now constructed an embedding $\varphi: (L_1,v_1) \longrightarrow (L_2,v_2)$ over $K$ which induces
$\varphi_k$ and $\varphi_\rmGamma$. Last but not least, if both $\varphi_k$ and $\varphi_\rmGamma$ are elementary embeddings,
the AKE$_{\preceq}$-principle (Corollary \ref{cor:AKE-E}) implies that $\varphi$ is in fact an elementary embedding
of unramified henselian valued fields, as desired. This finishes the proof.
\end{proof}

\begin{theorem} \label{thm:AKEE}
Let $(K,v) \subseteq (L,w)$ be an extension of unramified henselian valued fields
such that $Kv$ and $Lw$ have the same finite degree of imperfection.
If the induced embeddings of residue field $Kv \subseteq Lw$ and value group
$vK \subseteq wL$ are existentially closed (in $\mathfrak{L}_\textrm{ring}$ and $\mathfrak{L}_\textrm{oag}$ respectively), we have
$$(K,v) \preceq_\exists (L,w).$$
\end{theorem}
\begin{proof}
By taking ultrapowers if necessary, we may assume that both $(K,v)$ and $(L,w)$ are $\aleph_{1}$-saturated.
Let $(K^{*},v^{*})$ be a $|L|^{+}$-saturated elementary extension of $(K,v)$.
Since $Kv\preceq_{\exists}Lw$ and $vK\preceq_{\exists}wL$, there are embeddings
$\varphi_{k}:Lw\longrightarrow (Kv)^{*}=K^{*}v^{*}$ over $Kv$
and
$\varphi_{\rmGamma}:wL\longrightarrow (vK)^{*}=v^{*}K^{*}$ over $vK$.
Let $\bfbeta\subseteq Kv$ be a $p$-basis.
Since $Kv\preceq_{\exists}Lw$, $\bfbeta$ is also $p$-independent in $Lw$;
in particular $Lw/Kv$ is separable.
Since the degree of imperfection of $Lw$ is the same as that of $Kv$, $\bfbeta$ is a $p$-basis of $Lw$.
Since $K^{*}v^{*}$ is an elementary extension of $Kv$, $\bfbeta$ is also a $p$-basis of $K^{*}v^{*}$.
Therefore $K^{*}v^{*}/\varphi_{k}(Lw)$ is separable.
Finally, since $vK\preceq_{\exists}wL$, the group $wL/vK$ is torsion-free.
Thus we may apply Proposition~\ref{prp:emb} to obtain an embedding
$\varphi:(L,w)\longrightarrow(K^{*},v^{*})$ over $K$ that induces $\varphi_{k}$ and $\varphi_{\rmGamma}$.
In particular, $(K,v)\preceq_{\exists}(L,w)$.
\end{proof}

Note that the reason why we require a fixed finite degree of imperfection in Theorem \ref{thm:AKEE} is that when we embed $Lw$ into an elementary extension of $Kv$ in the proof, we need to make sure that this latter extension is separable in order to apply Proposition \ref{prp:emb}.

As a consequence of Theorem \ref{thm:AKEE}, we get the following existential version of the Ax-Kochen/Ershov result:
\begin{corollary}\label{cor:AKE-EE}
Let $(K,v)\subseteq (L,w)$ be two unramified henselian valued fields
such that $Kv$ and $Lw$ have the same finite degree of imperfection.
Then, we have
$$ \underbrace{Kv \preceq_\exists Lw}_{\textrm{in }\mathfrak{L}_\mathrm{ring}} \textrm{ and } \underbrace{vK \preceq_\exists wL}_{\textrm{in }\mathfrak{L}_\mathrm{oag}} \, \Longleftrightarrow \, 
\underbrace{(K,v) \preceq_\exists (L,w)}_{\textrm{in }\mathfrak{L}_\mathrm{vf}}.
$$
\end{corollary}

\section{Stable embeddedness}
In this final section, we comment on the structure induced on the residue field
and value group in unramified henselian valued fields. 

\begin{definition}
Let $\mathcal{M}$ be a structure and let $P\subseteq M^{k}$ be a definable set.
We say that $P$ is
{\em stably embedded}
if for all formulas
$\varphi(x,y)$
and all $b\in M^{|y|}$,
$\varphi(M^{k|x|},b)\cap P^{|x|}$ is $P$-definable.
\end{definition}

In order to be able to consider the residue field and the value group as definable sets in
a valued field (respectively, in a non-strict Cohen ring), it is necessary to switch from the language
of valued fields (respectively, the language of rings) to a multisorted setting.
Neither the results nor the proofs in this section are sensitive to this technicality.
An alternative approach is to widen the definition of stable embeddedness to also apply to interpretable sets, as discussed in
\cite{Pillay}.

In the appendix to \cite{CH}, Chatzidakis and Hrushovski show stable embeddedness of a type-definable set is equivalent to an automorphism-lifting criterion in saturated models. 
This shows:
\begin{theorem}\label{thm:SEk_non-strict}
Let $A\models T_{\mathrm{pc}}(k,n)$.
Then $k_{A}$ is stably embedded in $A$.
\end{theorem}
\begin{proof}
By the Structure Theorem (Corollary \ref{cor:Cohen_structure_2}),
each automorphism of the residue field $k_{A}$ lifts to an automorphism of the (non-strict) Cohen ring $A$.
In particular this holds for sufficiently saturated models, and thus $k_A$ is stably embedded as a
pure field by \cite[Appendix. Lemma 1]{CH}.
\end{proof}

Our aim is now to show that both residue field and value group are stably embedded
in unramified henselian valued fields. Following the approach in \cite[Lemma 3.1]{JS20}, we do this via our embedding lemma. 
More precisely, stable embeddedness follows from Proposition \ref{prp:emb} in a straightforward manner
since the type of an element of the residue field (respectively, the value group)
over the residue field (resp., value group) of an elementary submodel determines the type of that element over the submodel:
\begin{theorem} \label{thm:SEk}  \label{thm:stably.embedded}
Let $(K,v)$ be an unramfied henselian valued field. Then the value group $vK$ and the residue field $Kv$ are both stably embedded, as a pure ordered abelian group and as a pure field, respectively.
\end{theorem}
\begin{proof}
In light of Proposition \ref{prp:emb}, this is exactly the same argument as the proof of \cite[Lemma 3.1]{JS20}.
\end{proof}

As an immediate consequence, we get the following generalization of \cite[Theorem 7.3]{vdD14} to the case of imperfect residue fields:
\begin{corollary}
Let $(K,v)$ be an unramified henselian valued field. Then each subset of $Kv^n$ which is $\mathfrak{L}_\mathrm{vf}$-definable
in $(K,v)$ is already $\mathfrak{L}_\mathrm{ring}$-definable in $Kv$.
\end{corollary}

We now give an example that stable embeddedness of the residue field no longer
holds for finite extensions of unramified henselian valued fields.
A valued field $(K,v)$ of mixed characteristic $(0,p)$ is {\bf finitely ramified} if the interval $(0,v(p)]$ is finite.
The following example shows that an analogue of Theorem \ref{thm:SEk}
does not hold for all finitely ramified henselian valued fields.

\begin{example} \label{ex:fr}
Consider a field $k$ of characteristic $p>2$ with elements $\alpha_{1},\alpha_{2}\in k$ such that
\begin{enumerate}
\item
there is an automorphism $\varphi$ of $k$ which maps $\alpha_{1}$ to $\alpha_{2}$, and
\item
$\alpha_{1}$ and $\alpha_{2}$ lie in different multiplicative cosets of $k^{\times2}$.
\end{enumerate}
Let $(K,v)$ be an unramified henselian valued field with residue field $k$.
We distinguish a representative $a_{1}\in\res^{-1}(\alpha_{1})$ of $\alpha_{1}$, and let
$(L,w)$ be the extension of $(K,v)$ given by adjoining a square-root $b_{1}$ of $pa_{1}$.
Then $(L,w)$ is henselian and finitely ramified, and $(L,w)/(K,v)$ is a quadratic extension with ramification degree $e=2$ and inertia degree $f=1$.
It is also easy to see that $L$ contains no element $b_{2}$ such that
$b_{2}^{2}= pa_{2}$,
where $a_{2}$ is any representative of $\alpha_{2}$.
Suppose that $\rmPhi$ is an automorphism of $(L,w)$ which lifts $\varphi$.
Then
\begin{align*}
    \res(\rmPhi(b_{1})^2/p) = \res(\rmPhi(a_{1})) = \varphi(\res(a_{1})) = \alpha_{2},
\end{align*}
showing that $\rmPhi(b_{1})$ is just such a non-existent element $b_{2}$ of $L$,
which is a contradiction.

Note that there are elementary extensions $(L^{*},w^{*})$ of $(L,w)$ with an automorphism
$\varphi^*$ on the residue field $L^*w^*$ such that (i) and (ii) still hold.
These can be obtained by adding a symbol for the automorphism $\varphi$ to the language on the residue field before saturating.
The argument above shows in fact that in none of these models, $\varphi^*$ can be lifted to an automorphism of $L^*$.

Therefore the residue field $Lw=k$ is not stably embedded in $(L,w)$.
Finally, note that there are many such fields $k$;
for example consider $k=\mathbb{F}_{p}(\alpha_{1},\alpha_{2})$ where $\alpha_{1},\alpha_{2}$ are algebraically independent over $\mathbb{F}_{p}$.
\end{example}

\section*{Acknowledgements}

We would like to extend our sincere thanks to
Martin Hils, Florian Pop, Silvain Rideau-Kikuchi, and Tom Scanlon,
for numerous very helpful conversations.

This work was begun while we were participating in the
{\em  Model Theory, Combinatorics and Valued fields}
trimester at the Institut Henri Poincar\'{e}
and
the Centre \'{E}mile Borel,
and we would like to extend our thanks to the organisers.
Sylvy Anscombe was also partially supported by The Leverhulme Trust
under grant RPG-2017-179.
Franziska Jahnke was also
funded by the Deutsche Forschungsgemeinschaft (DFG, German Research Foundation) under Germany's Excellence Strategy EXC 2044-390685587, Mathematics M\"unster: Dynamics-Geometry-Structure and the CRC878, as well as by a Fellowship from the Daimler and Benz Foundation.

\bibliographystyle{plain}

\end{document}